\documentclass[10pt,reqno,a4paper,twoside]{extarticle}

\usepackage{charter}

\usepackage[left=1in,right=1in,top=1in,bottom=1in]{geometry}

\usepackage{graphicx, varioref, amsfonts}
\usepackage{amsthm,amssymb,amsmath,mathrsfs,mathtools,stmaryrd} %\usepackage{stmaryrd}
\usepackage{dsfont}
\usepackage{upgreek} %% txgreeks
\usepackage{esint} %% nice integral symbols 
\usepackage{breakurl}
\usepackage{empheq}
\usepackage{enumerate,enumitem}
\usepackage{hyperref}
\usepackage{cleveref}
\usepackage{autonum}
\usepackage{soul}
\usepackage{cancel}
\usepackage{amscd}
\usepackage{srcltx}
\usepackage{ifthen}
\usepackage{float}
\usepackage{color}
\usepackage{comment}
\usepackage{chngcntr}
\usepackage{etoolbox}
\usepackage{fancyhdr}
\usepackage{breqn}
\usepackage{cite}
\usepackage[T1]{fontenc}%spelling swiech
\usepackage{ulem}
\usepackage{xcolor}

\usepackage{titlefoot}

\hypersetup{
    colorlinks=true,
    linkcolor=blue,
    filecolor=blue,      
    urlcolor=blue,
    citecolor=blue,
    pdftitle={Overleaf Example},
    pdfpagemode=FullScreen,
    }

\usepackage{authblk} % For \affil command
\theoremstyle{plain}
\newtheorem{lemma}{Lemma}[section]

\newtheorem{proposition}[lemma]{Proposition}

\newtheorem{theorem}[lemma]{Theorem}

\newtheorem{assumption}[lemma]{Assumption}

\theoremstyle{definition}

\newtheorem{definition}[lemma]{Definition}

\newtheorem{example}[lemma]{Example}
\newtheorem{remark}[lemma]{Remark}
\newcommand{\R}{\ensuremath{\mathbb R}} 
\DeclareMathOperator*{\argmax}{arg\,max}

\newlist{todolist}{itemize}{2}
\setlist[todolist]{label=$\square$}
\usepackage{pifont}

\numberwithin{equation}{section}
\begin{document}

\title{Mean Field Games in Hilbert Spaces with Degenerate Diffusion: A Viscosity Solution Approach
}
\newcommand\shorttitle{Mean Field Games in Hilbert Spaces}
\date{May 12, 2026}
%\date{\today}

\author{Andrzej {\'{S}}wi{\k{e}}ch\footnote{School of Mathematics, Georgia Institute of Technology, Atlanta, GA 30332, USA; Email: swiech@math.gatech.edu}}
\author{Lukas Wessels\footnote{Laboratoire J.A. Dieudonn\'e, Universit\'e C\^ote d’Azur, 06108 Nice, France; Email: lukas.wessels@univ-cotedazur.fr}}
\newcommand\authors{Andrzej {\'{S}}wi{\k{e}}ch and Lukas Wessels}

\affil{}%\small Technische Universit\"at Berlin}

\maketitle

\unmarkedfntext{\textit{Mathematics Subject Classification (2020) ---} 35D40, 35Q84, 35Q89, 35R15, 49L12, 49L25, 49N80, 91A16}

%28A33: Spaces of measures, convergence of measures
%35D40: Viscosity solutions to PDEs
%35F21: Hamilton--Jacobi equations
%35K57: Reaction-diffusion equations
%35Q84: Fokker-Planck equations
%35Q89: PDEs in connection with mean field game theory
%35R15: PDEs on infinite-dimensional (e.g., function) spaces (= PDEs in infinitely many variables
%49K27: Optimality conditions for problems in abstract spaces
%49K45: Optimality conditions for problems involving randomness
%49L12:	Hamilton--Jacobi equations in optimal control and differential games
%49L20: Dynamic programming in optimal control and differential games
%49L25:	Viscosity solutions to Hamilton--Jacobi equations in optimal control and differential games
%49N35:	Optimal feedback synthesis
%49N80: Mean-field games and control
%60H15: SPDEs (aspects of stochastic analysis)
%65K10: Numerical optimization and variational techniques
%91A16: Mean field games (aspects of game theory)
%93E20: Optimal stochastic control

\unmarkedfntext{\textit{Keywords and phrases ---} Mean field games, PDEs in infinite dimensional spaces, Hamilton--Jacobi--Bellman equation, Fokker--Planck equation, Viscosity solutions, Wasserstein space, Stochastic differential equations in infinite dimension.}

%\unmarkedfntext{\textit{Email}: \textbullet$\,$ TBD $\,$\textbullet$\,$ TBD $\,$\textbullet$\,$ TBD}

\begin{abstract}
%We study a mean field game (MFG) 
%in which the state dynamics are governed by a stochastic differential equation in a Hilbert space $H$ with degenerate noise. The associated 
We study a degenerate second order mean field game (MFG) system in a Hilbert space $H$ which couples a Fokker--Planck equation describing the evolution of probability measures on $H$ with a Hamilton--Jacobi--Bellman (HJB) equation for the value function. Our main result establishes existence and uniqueness of solutions to this coupled system. Solutions of the HJB equation are interpreted in the viscosity sense. For existence, we extend the classical fixed-point approach based on Tikhonov’s theorem to our setting. A central difficulty in this approach is proving uniqueness for the corresponding linear degenerate Fokker--Planck equation. To address this issue, we introduce a class of suitable adjoint equations and employ viscosity solution techniques to construct sufficiently regular solutions. Uniqueness for the full MFG system is then obtained via an adaptation of the Lasry--Lions monotonicity method.
\end{abstract}

%\tableofcontents

\section{Introduction}

In this paper, we study a class of mean field game systems formulated on a Hilbert space. Such systems arise as the limit of $N$-player stochastic differential games as $N\to\infty$, where the dynamics of each agent are governed by an infinite-dimensional stochastic differential equation (SDE) involving unbounded operators. Our framework encompasses important examples, including stochastic delay differential equations and stochastic partial differential equations.

More precisely, let $H$ be a real, separable Hilbert space, and denote by $\mathcal{P}_1(H)$ the Wasserstein space of order $1$ over $H$. We consider the following coupled system of partial differential equations:
\begin{align}
    \partial_t v(t,x) + Lv(t,x) - \mathcal{H}(x,Dv(t,x),m(t))& = 0, & v(T,\cdot) &= G(\cdot,m(T)) \label{HJB} \tag{HJB} \\
    \partial_t m(t) - L^* m(t) - \operatorname{div} \left ( \mathcal{H}_p(x,Dv(t,x),m(t)) m(t) \right ) &= 0,& m(0) &= m_0 \in \mathcal{P}_1(H), \label{FP} \tag{FP}
\end{align}
where $v:[0,T]\times H \to \mathbb{R}$ is the value function of a representative agent and $m : [0,T] \to \mathcal{P}_1(H)$ describes the distribution of the population. For a function $\phi:H \to \mathbb{R}$, the operator $L$ is formally defined by
\begin{equation}\label{operator_L}
    L\phi(x) := \langle Ax, D\phi(x) \rangle + \frac12 \operatorname{Tr}[ \sigma(x) \sigma^*(x) D^2 \phi(x) ]
\end{equation}
where $A$ is an unbounded operator on $H$. Above, $L^*$ is the formal adjoint of $L$ and $\operatorname{div}$ is another formal operator representing the ``negative of the adjoint of the gradient.'' The precise meaning of these operators will be explained in Section \ref{sec:LFP} and our assumptions on the coefficients are given in Section \ref{sec:notass}. We remark that in this paper we write \eqref{HJB} and other terminal value problem HJB equations using the convention of \cite{fabbri_gozzi_swiech_2017} rather than the perhaps more customary one with the minus sign in front of $\partial_t v$, (which was also used in \cite{federico_gozzi_swiech_2026}).

The main novelty and difficulty in the study of this MFG system stems from the infinite dimensionality of the underlying state space $H$ combined with the unboundedness of the operator $A$ and the degeneracy of the diffusion coefficient $\sigma$ in \eqref{operator_L}. Our main contribution is the proof of existence and uniqueness of solutions to the MFG system \eqref{HJB}--\eqref{FP} in the framework where the HJB equation is interpreted in the viscosity sense. Along the way, we establish uniqueness of weak solutions to a linear degenerate Fokker--Planck equation with our operator $L^*$. This is a new result which is of independent interest. Our results also hold for first order MFG systems. To maintain a clear focus on the main ideas and core analytical arguments, we adopt relatively strong assumptions on the Hamiltonian $\mathcal{H}$ (see Assumptions \ref{ass:HJB}(A2)--(A4)). This may narrow the class of HJB equations having the required regularity of viscosity solutions and hence restrict the scope of applications. However, our main goal is to show that the viscosity solution approach is viable, and we anticipate that the methodology developed here is robust enough to be extended to a broader class of Hamiltonians. In particular, we think that essentially the same approach should work in the case where the estimates in Assumptions \ref{ass:HJB}(A2),(A3) are more local and $\sigma$ is unbounded.

\iffalse
Moreover,
\begin{enumerate}[label=(\roman*)]
    \item $G:H\times \mathcal{P}_1(H) \to \mathbb{R}$ denotes the terminal cost;
    \item $L\phi(x) = \langle Ax, D\phi(x) \rangle + \frac12 \operatorname{Tr}[ \sigma(x) \sigma^*(x) D^2 \phi(x) ]$ for $\phi:H \to \mathbb{R}$, where $A$ is an unbounded operator on $H$ and $\sigma: H \to L_2(\Xi,H)$ maps into the space of Hilbert--Schmidt operators from some real, separable Hilbert space $\Xi$ into $H$. $L^*$ is the adjoint of $L$; \label{ii}
    \item $\mathcal{H}:H\times H\times \mathcal{P}_1(H) \to \mathbb{R}$ denotes the Hamiltonian and $\mathcal{H}_p$ denotes its Fr\'echet derivative with respect to the second argument.
\end{enumerate}

where
\begin{enumerate}[label=(\roman*)]
    \item $m_0\in \mathcal{P}_1(H)$, $G:H\times \mathcal{P}_1(H) \to \mathbb{R}$;
    \item $L$ is the operator formally defined by acting on a function $\phi :H\to \mathbb{R}$ via
    \begin{equation}
        L\phi(x) = \langle Ax, D\phi(x) \rangle + \frac12 \operatorname{Tr}[ \sigma(x) \sigma^*(x) D^2 \phi(x) ],
    \end{equation}
    where $A:\mathcal{D}(A) \subset H \to H$ is a closed, densely defined linear operator and $\sigma: H \to L_2(\Xi,H)$; \label{ii}
    \item $L^*$ is, formally, the adjoint of $L$ and the operator $\operatorname{div}$ is, formally, the negative of the adjoint of the gradient;
    \item $\mathcal{H}:H\times H\times \mathcal{P}_1(H) \to \mathbb{R}$, and $\mathcal{H}_p$ denotes the Fr\'echet derivative of $\mathcal{H}$ with respect to the second argument.
\end{enumerate}
\fi

In the finite dimensional setting, mean field games were introduced simultaneously and independently by Lasry and Lions, and Huang, Caines and Malham{\'e} in \cite{huang_caines_malhame_2007,lasry_lions_2006,lasry_lions_2006_2,lasry_lions_2007}. Since these seminal works, the theory has been the subject of extensive study and the literature has expanded rapidly. For comprehensive introductions and overviews, we refer to \cite{carmona_delarue_2018,carmona_delarue_2018_2,cardaliaguet_delarue_lasry_lions_2019, cardaliaguet_porretta_2020,gomes_saude_2014}. By contrast, the literature on mean field games in infinite dimensions and on mean field games with degenerate diffusion remains relatively limited.

An early contribution to mean field games in infinite dimensions is \cite{fouque_zhang_2018}, where a linear-quadratic mean field game with delay is studied. Using a lifting procedure, the problem is reformulated as an infinite-dimensional MFG. The authors derive the associated master equation and obtain an explicit solution, which relies crucially on the linear–quadratic structure of the model. More recently, the linear-quadratic case has been further investigated in \cite{federico_ghilli_gozzi_2025,firoozi_kratsios_yang_2025,ghilli_ricciardi_2025,liu_firoozi_2025,liu_firoozi_2026}. To the best of our knowledge, the only work that goes beyond the linear-quadratic case is \cite{federico_gozzi_swiech_2026}, where existence and uniqueness are established for an MFG system similar to \eqref{HJB}--\eqref{FP}. The analysis in \cite{federico_gozzi_swiech_2026} is carried out in the framework of mild solutions, which requires strong assumptions on the diffusion coefficient $\sigma$ and the unbounded operator $A$ in order to obtain suitable smoothing properties of the associated Ornstein--Uhlenbeck semigroup. In particular, these assumptions are not satisfied in the setting considered here, so their approach does not apply to our problem. We will discuss the relationship between their results and ours in more detail at the end of the introduction below.

In the finite-dimensional setting, there exists an extensive literature concerning first order MFG systems, see e.g. \cite{alharbi_gomes_2026,cardaliaguet_2015,cardaliaguet_graber_2015,cardaliaguet_meszaros_santambrogio_2016,cardaliaguet_porretta_2020,cirant_nurbekyan_2018,graber_2014,griffin-pickering_meszaros_2022,munoz_2022,santambrogio_2020} and the references therein. In particular, the classical vanishing viscosity method to obtain solutions to first order MFG systems is described in \cite{cardaliaguet_porretta_2020}. In contrast, the theory for degenerate second order systems has received significantly less attention. To the best of our knowledge, the first work addressing such a problem is \cite{cardaliaguet_graber_porretta_tonon_2015}, where existence and uniqueness of suitably defined weak solutions are established. Their analysis exploits the fact that the MFG system can be interpreted as the optimality condition for two convex problems. Subsequently, in \cite{ferreira_gomes_2018,ferreira_gomes_tada_2021}, the authors construct solutions to degenerate MFG systems via a higher-order elliptic regularization: First, they consider non-degenerate approximating problems, and then prove the existence of weak solutions to the original system using monotonicity methods. More recently, in \cite{chowdhury_jakobsen_krupski_2025}, the authors study a class of degenerate MFG systems imposing a certain structure on the system that still guarantees the existence of a classical solution to the HJB equation.

Our general approach follows that of \cite{federico_gozzi_swiech_2026} and is based on the classical strategy using Tikhonov's fixed point theorem. The main difference is that we use the theory of viscosity solutions. The key difficulty to overcome here is proving uniqueness of solutions to the degenerate Fokker--Planck equation \eqref{FP} when the term $\mathcal{H}_p(x,Dv(t,x),m(t))$ is fixed, thereby reducing the equation to a linear form. Linear Fokker--Planck equations in infinite-dimensional spaces have been studied extensively, see e.g.\cite{bogachev_daprato_roeckner_2009,bogachev_daprato_roeckner_2010,bogachev_daprato_roeckner_2011,daprato_flandoli_roeckner_2013,roeckner_zhu_zhu_2014,wiesinger_2013}. While the works \cite{bogachev_daprato_roeckner_2009,bogachev_daprato_roeckner_2011,wiesinger_2013} address the uniqueness of solutions of such equations with degenerate diffusion, they all work, among other conditions, under the assumption that $\sigma(x) \equiv \sigma$ is independent of $x\in H$ and $\mathrm{e}^{tA} (H) \subset ( \int_0^t \mathrm{e}^{sA} \sigma \sigma^* \mathrm{e}^{sA^*} \mathrm{d}s )^{1/2}(H)$, which severely restricts the choice of the unbounded operator $A$ when the noise is degenerate.

We now provide a more detailed account of our results and methodology. As described above, the first step involves solving the linear Fokker--Planck equation \eqref{FP} when $\mathcal{H}_p(x,Dv(t,x),m(t))$ is fixed; this analysis is the subject of Section \ref{sec:LFP}. To prove uniqueness of weak solutions (see Definition \ref{definition_weak_solution} for the precise meaning) we employ a classical approach of constructing sufficiently smooth solutions for the adjoint equation (see equation \eqref{equation_1} below) for a broad class of right-hand sides. This adjoint equation is a linear degenerate second order equation. Due to the degeneracy, classical or sufficiently regular solutions are generally not available. Instead, we work with viscosity solutions. In order to obtain the necessary regularity, we employ three approximation procedures: (i) replacing the unbounded operator with its Yosida approximation, (ii) projecting onto finite-dimensional subspaces, and (iii) applying a classical finite-dimensional mollification. The existence of a solution for the linear Fokker--Planck equation is then obtained via the corresponding SDE.

Once the well-posedness of the linear Fokker–Planck equation is established, we turn to the well-posedness of the MFG system \eqref{HJB}--\eqref{FP}. A natural way to define generalized solutions $(v,m)$ to the MFG system is to require that $v$ is a viscosity solution of the HJB equation and has a Fr\'echet derivative in the spatial variable that is Lipschitz continuous, and that $m$ is a weak solution of the linear Fokker--Planck equation obtained from equation \eqref{FP} when $\mathcal{H}_p(x,Dv(t,x),m(t))$ is fixed. However, here we require a little more, namely that $v(t,\cdot)\in C^{1,1}(H_{-1})$, $t\in [0,T]$, where $H_{-1}$ is another Hilbert space such that $H$ is compactly embedded in $H_{-1}$. See the beginning of Section \ref{sec:notass} for the precise construction of $H_{-1}$ and see Definition \ref{definition_MFG_solution} for our definition of a solution to the MFG system. We remark that the existence of viscosity solutions $v$ of HJB equations with the required regularity was established in \cite{defeo_swiech_wessels_2025} under certain assumptions; accordingly, we impose the existence of such solutions as an assumption in the present paper (see Section \ref{section_C11_regularity} for details). In Section \ref{sec:existence}, solutions are then constructed via Tikhonov’s fixed point theorem, which relies on compactness. In infinite-dimensional spaces, compactness is significantly more difficult to obtain; in our setting, it is recovered by working in the space $H_{-1}$.

In Section \ref{section_MFG_Uniqueness}, the uniqueness of solutions is established via an argument based on the Lasry--Lions monotonicity condition and separability assumptions. In finite-dimensional settings, this approach is classical. For sufficiently regular solutions, the argument proceeds as follows: given two solutions $(v^1,m^1)$ and $(v^2,m^2)$ of the MFG system, one uses $v^1-v^2$ as a test function for the equation satisfied by $m^1-m^2$, and concludes uniqueness under suitable structural assumptions. In the present framework, however, the low regularity of MFG solutions prevents a direct application of this method. To overcome this difficulty, we perform the above computation at the approximate level, and then pass to the limit. In this respect we follow \cite{federico_gozzi_swiech_2026}. However, instead of deriving the Kolmogorov equation that is satisfied by $v^1-v^2$ and approximating its solution, we use the regularization procedure that we already employed in the proof of uniqueness for the Fokker--Planck equation and apply it separately to two solutions $v^1$ and $v^2$ of the original HJB equation. Uniqueness of the MFG system is then recovered by following the same strategy as in the finite-dimensional, smooth setting. 

In Sections \ref{sec:examples} and \ref{section_optimal_advertising}, we discuss various examples. In particular, as an application, we analyze an optimal advertising problem in a large market that incorporates memory effects.

Finally, we compare our results in more detail with those of \cite{federico_gozzi_swiech_2026}. In that work, the analysis is carried out in the framework of mild solutions under the assumptions that $\sigma = I$ and that $A$ is a closed, densely defined, negative, self-adjoint operator such that $(-A)^{-1+\delta}$ is trace class for some $\delta > 0$. These conditions are used to apply the uniqueness result of \cite{bogachev_daprato_roeckner_2011} for weak solutions to the linear Fokker--Planck equation, to obtain the required smoothing properties of the associated Ornstein--Uhlenbeck semigroup, and to ensure compactness of certain sets of measures arising in the analysis. In contrast, we work in the framework of viscosity solutions. Here, $\sigma$ is allowed to be degenerate and $A$ can, in principle, be any maximal dissipative operator. Thus, our assumptions are of a different nature and apply to a broader class of models. However, the two sets of assumptions are essentially disjoint, and the results of \cite{federico_gozzi_swiech_2026} and the present work should be viewed as complementary. The use of mild solutions in \cite{federico_gozzi_swiech_2026} allows for minimal assumptions on the Hamiltonian $\mathcal{H}$. By contrast, the viscosity approach requires working in the weaker space $H_{-1}$, defined via a suitable compact operator $B$ associated with $A$. This leads to continuity requirements on $\mathcal{H}$ with respect to the $H_{-1}$-norm and imposes some restrictions on the operator $A$. In our setting, compactness is recovered by working in the space $\mathcal{P}_1(H_{-1})$. Moreover, we require that viscosity solutions $v$ of \eqref{HJB} satisfy $v(t,\cdot) \in C^{1,1}(H_{-1})$. This appears to be a minimal regularity condition for the applicability of the viscosity framework, and it is satisfied in a number of relevant situations. Overall, the results of the present work significantly advance the understanding of MFG systems in Hilbert spaces.

\section{Notation and Assumptions}\label{sec:notass}

Throughout the manuscript, $H$ is a real, separable Hilbert space with inner product $\langle \cdot,\cdot \rangle$ and norm $\|\cdot\|$, and $A:\mathcal{D}(A) \subset H \to H$ is a linear, closed, densely defined maximal dissipative operator. Moreover, $B:H\to H$ is a linear, bounded, self-adjoint, strictly positive, compact operator such that $A^*B$ is bounded. We identify $H$ with its dual. We use the following notation.
\begin{itemize}
 \item For real, separable Hilbert spaces $V$ and $W$, $L(V,W)$ denotes the space of bounded linear operators from $V$ to $W$ equipped with the operator norm $\|\cdot\|_{L(V,W)}$. If $V=W$ we will just write $L(V)$. We write $L_2(V,W)$ to denote the space of Hilbert--Schmidt operators in $L(V,W)$ and use $\|\cdot\|_{L_2(V,W)}$ to denote its norm. We denote by $S(V)$ the space of self-adjoint operators in $L(V)$. We denote by $A_n = nA(nI-A)^{-1}$, $n=1,2,\dots$, the Yosida approximations of $A$.

 For a Fr\'echet differentiable function $\varphi: V \to \mathbb{R}$, we denote by $D\varphi(x)$ its derivative at $x$, identified with the gradient in $V$ via the Riesz representation theorem. We write $D_V \varphi(x)$ to emphasize the space and the corresponding inner product used to define this gradient.
  
\item We define the space $H_{-1}$ as the completion of $H$ with respect to the norm $\|x\|_{-1}^2 := \langle Bx,x\rangle$.  $H_{-1}$ is a Hilbert space with the inner product $\langle x,y\rangle_{-1}:= \langle B^{1/2}x,B^{1/2}y\rangle$. For more details, see \cite[Section 3.1.1]{fabbri_gozzi_swiech_2017}. There exists an orthonormal basis $(e_k)_{k\in \mathbb{N}}$ of $H$ consisting of eigenvectors of $B$. We arrange them in order of decreasing eigenvalues $(\lambda_k)_{k\in\mathbb{N}}$. We denote $H_N={\rm span}\{e_1,\dots, e_N\}$ and let $P_N$ be the orthogonal projection in $H$ onto $H_N$. We denote $Q_N=I-P_N$ and $H_N^\perp=Q_N(H)$. For $x\in H$ we will write $x=(x_N,x_N^\perp)$ and for $x_N\in H_N$ we will sometimes write $x_N=(x_N,0)$. Since $H_N$ and $H_N^\perp$ are also orthogonal in $H_{-1}$, $P_N$ and $Q_N$ extend to orthogonal projections in $H_{-1}$ onto $H_N$ and the completion of $H_N^\perp$. We will also denote these projections in $H_{-1}$ by $P_N$ and $Q_N$. For $n,N\in \mathbb{N}$, we define the operators $A_{n,N}=P_NA_nP_N$. 
If $X\in S(H_N)$ we may identify it with the operator $XP_N=P_NXP_N\in S(H)$, that is we consider it as the operator in $S(H)$ in the block form
\[
\begin{bmatrix}
X & 0\\
0 & 0
\end{bmatrix}.
\]
With this identification, if $\varphi\in C^2(H_N)$, then the function $\psi(x)=\varphi(P_Nx)$ is in $C^2(H)$ and $D^2\psi(x)=P_ND^2\varphi(P_Nx)P_N$, that is
\[
D^2\psi(x)=
\begin{bmatrix}
D^2\varphi(P_Nx) & 0\\
0 & 0
\end{bmatrix}.
\]
Moreover, if $\varphi\in C^2(H)$, then $D^2_{x_N}\varphi(x)=P_ND^2\varphi(x)P_N$, which when restricted to $H_N$ is an operator in $S(H_N)$.

    \item We denote by $\mathcal{P}(H)$ the set of all Borel probability measures on $H$; for $p\in [1,\infty)$, we denote by $\mathcal{P}_p(H)$ the set of all $\mu\in \mathcal{P}(H)$ with finite $p$-th moment, i.e., $\int_{H} \|x\|^p \mu(\mathrm{d}x) < \infty$. Let $\mathbf{d}_{p}:\mathcal{P}_p(H) \times \mathcal{P}_p(H) \to \mathbb{R}$ denote the $p$-th Wasserstein distance, i.e.,
    \begin{equation}\label{definition_wasserstein_distance}
        \mathbf{d}_p(\mu,\nu) = \inf_{\gamma \in \Gamma(\mu,\nu)} \left ( \int_{H\times H} \| x-y \|^p \gamma(\mathrm{d}x,\mathrm{d}y) \right )^{\frac1p},
 \end{equation}
    where $\Gamma(\mu,\nu)$ is the set of all Borel probability measures on $H \times H$ with first and second marginals $\mu$ and $\nu$, respectively. 
    We notice that if $(\Omega, \mathcal{F},\mathbb{P})$ is a probability space then
    \[
   \mathbf{d}_p(\mu,\nu)\leq \inf \left \{ \left ( \int_{\Omega} \|X(\omega)-Y(\omega)\|^p \mathrm{d}\omega \right )^{\frac1p}: \; X,Y\in L^p(\Omega;H), \; {\mathcal L}(X) = \mu, \; {\mathcal L}(Y)= \nu \right \}.
   \]
   Analogously, $\mathbf{d}_{p,-1}: \mathcal{P}_p(H_{-1}) \times \mathcal{P}_p(H_{-1}) \to \mathbb{R}$ denotes the $p$-th Wasserstein distance on $\mathcal{P}_p(H_{-1})$, i.e., the space $H$ and its norm above are replaced by the space $H_{-1}$ and its norm, and $\Gamma(\mu,\nu)$ is replaced by the set of $\Gamma_{-1}(\mu,\nu)$ of Borel probability measures on $H_{-1} \times H_{-1}$ with marginals $\mu$ and $\nu$. We denote by $\mathcal{M}(H_{-1})$ the space of finite signed Borel measures on $H_{-1}$ endowed with the weak topology. We remark that every measure in $\mathcal{P}_p(H)$ or in $\mathcal{P}_p(H_{-1})$ is tight and has compact inner approximation property for every Borel set, see e.g. \cite[Chapter II, Theorems 3.1 and 3.2]{parthasarathy_1967}. Also, since $\mathcal{B}(H_{-1}) \cap H = \mathcal{B}(H)$, see \cite[Lemma 1.17]{fabbri_gozzi_swiech_2017}, every $\mu\in \mathcal{P}_1(H)$ can be considered an element of $\mathcal{P}_1(H_{-1})$ by setting $\mu(F)= \mu(F \cap H)$ for $F \in \mathcal{B}(H_{-1})$.
   
   \item
   We define the sets $S,\hat S, \tilde S$:
\[
S:=\bigg\{m:[0,T]\to \mathcal{P}_1(H): m\in C([0,T];\mathcal{P}_1(H_{-1}))\,\,\text{and}\,\,\sup_{t\in[0,T]}\int_H\|x\| m(t,\mathrm{d}x)<+\infty\bigg\},
\]
 \[
  \hat S:=C([0,T];(\mathcal{P}_1(H), {\mathbf{d}_{1}})),\quad \tilde S:=C([0,T];(\mathcal{P}_1(H_{-1}), {\mathbf{d}_{1,-1}})).
  \]
The sets $\hat S, \tilde S$ are complete metric spaces equipped, respectively, with the metrics
\[
\rho_{1,\infty}(m_1,m_2):=\max_{t\in[0,T]}{\mathbf{d}_{1}}(m_1(t),m_2(t)),\quad \rho_{-1,\infty}(m_1,m_2):=\max_{t\in[0,T]}{\mathbf{d}_{1,-1}}(m_1(t),m_2(t)).
\]
We notice that if $m\in C([0,T];\mathcal{P}_1(H))$ and $t_n\to t$ then $m(t_n)$ converges weakly to $m(t)$ so, since $C_b(H_{-1})\subset C_b(H)$, $m(t_n)$ converges weakly to $m(t)$ when measures are considered to be elements of $\mathcal{P}_1(H_{-1})$. It then easily follows from \cite[Proposition 7.1.5]{ambrosio_gigli_savare_2008} that $\mathbf{d}_{1,-1}(m(t_n),m(t))\to 0$. Hence, $m\in C([0,T];\mathcal{P}_1(H_{-1}))$.

We also notice that if $m\in S$ then $m(t)\mathrm{d}t$ is a Borel measure on $[0,T]\times H_{-1}$ and hence also on $[0,T]\times H$. To show this, let $m\in S$. We first repeat the argument from \cite{federico_gozzi_swiech_2026} that $t\mapsto m(t)(F)$ is measurable for each $F \in B(H_{-1})$. Indeed, this map is the composition of the map $t\mapsto m(t)$ from $[0,T]$ to $(\mathcal{P}_1(H_{-1}),\mathbf{d}_{1,-1})$ and the map $\mu\mapsto \mu(F)$ which we consider as a map from $(\mathcal{P}_1(H_{-1}), \text{weak topology})$ to $\mathbb{R}$. The first map is continuous by definition of $S$ and, since the $\mathbf{d}_{1,-1}$-topology is stronger than the weak topology, it suffices to show that the second map is measurable. Since $\mu\mapsto \mu(F)$ is upper semicontinuous for every closed set $F$ (see e.g. \cite{ambrosio_gigli_savare_2008}, page 110), the measurability for every $F \in \mathcal{B}(H_{-1})$ is a consequence of \cite[Proposition 7.25]{bertsekas_shreve_1978}. Now let $G\in {\mathcal B}([0,T]\times H_{-1})$. We know that the function $t\to m(t)(G_t)$, where $G_t=\{x\in H_{-1}:(t,x)\in G\}$, is measurable for all rectangular sets $G=I\times F$, where $I\in \mathcal{B}(\mathbb R), F\in \mathcal{B}(H_{-1})$. Thus the measurability for any Borel set $G$ follows from the monotone class theorem. Therefore $\nu=m(t)\mathrm{d}t$ defined by
\[
\nu(G)=\int_0^Tm(t)(G_t) \mathrm{d}t
\]
is a Borel measure on $[0,T]\times H_{-1}$. Since $m\in S$ is supported on $[0,T]\times H$, $\nu$ is also a Borel measure on $[0,T]\times H$.

\item
For a set $\mathcal{D}$ and functions $v:\mathcal{D}\to \mathbb{R}$ and $\tilde{v}:\mathcal{D}\to H$, we denote (with a slight abuse of notation) $\|v\|_{\infty} = \sup_{x\in \mathcal{D}} |v(x) | $ and $\|\tilde{v}\|_{\infty} = \sup_{x\in \mathcal{D}} \| \tilde{v}(x) \|$, respectively. For a function $v:\mathcal{D}\to H_{-1}$, we denote $\|v\|_{-1,\infty} = \sup_{x\in \mathcal{D}} \| v(x) \|_{H_{-1}}$.

\item For a real, separable Hilbert space $(V,\|\cdot\|_V)$, we denote by $C^{1,1}_b(V)$ the space of bounded functions $v:V\to\mathbb{R}$ which are Fr\'echet differentiable and whose Fr\'echet derivatives are bounded and Lipschitz continuous. $C^{1,1}_b(V)$ is a Banach space equipped with the norm
  \[
  \|v\|_{C^{1,1}_b(V)}:=\|v\|_\infty+\|Dv\|_\infty+\sup_{x,y\in V,x\not= y}\frac{\|Dv(x)-Dv(y)\|_V}{\|x-y\|_V}.
  \]
  We denote 
  \[
  \begin{split}
C_{b}^{1,2}([0,T]\times V):=\Big\{\varphi\in C^{1,2}([0,T]\times V)&: \partial_t\varphi\in C_{b}([0,T]\times V), D\varphi\in C_{b}([0,T]\times V;V), 
\\
&D^2\varphi\in C_{b}([0,T]\times V;L(V))\Big\}.
\end{split}
\]
\end{itemize}

We use the notion of $B$-continuous viscosity solutions. This type of viscosity solution was first introduced for first order equations in \cite{crandall_lions_1990,crandall_lions_1991} and then generalized to second order equations in \cite{swiech_1994}. The full account of the basic theory of $B$-continuous viscosity solutions can be found in \cite[Chapter 3]{fabbri_gozzi_swiech_2017}. In particular, we refer to \cite[Definition 3.35]{fabbri_gozzi_swiech_2017} for the definition of the ($B$-continuous) viscosity solution used in this paper. The notion of $B$-continuity is also explained in detail in \cite[Section 3.1.1]{fabbri_gozzi_swiech_2017}. For the basic definition, see Appendix \ref{section_appendix}.

In addition to \eqref{HJB}, we consider the two approximating equations
\begin{equation}\label{HJBn}
    \partial_t v + \langle A_nx,Dv \rangle +\frac12 \operatorname{Tr} [ \sigma(x) \sigma^*(x) D^2 v ] - \mathcal{H}(x,Dv,m(t)) = 0,\quad v(T,\cdot) = G(\cdot,m(T)), 
\end{equation}
and
 \begin{equation}\label{HJBnN}
    \partial_t v + \langle A_{n,N} x,Dv \rangle + \frac12 \operatorname{Tr} [ \sigma(P_Nx) \sigma^*(P_Nx)P_N D^2 v P_N ] - \mathcal{H}(P_Nx,P_NDv,m(t)) = 0,\quad v(T,\cdot) = G(P_N\cdot,m(T)). 
\end{equation}

We impose the following assumptions.
\begin{assumption} 
\label{ass:HJB}
\begin{enumerate}[label=(A\arabic*)]
   \item
    The linear operator $A:\mathcal{D}(A) \subset H \to H$ is a closed, densely defined maximal dissipative operator, and the weak $B$-condition is satisfied, that is, there exist a strictly positive $B\in S(H)$ and $c_0\geq 0$ such that
    \[
        -A^* B+c_0 B \geq 0.
    \]
The operator $B$ is compact and $A^* B^{1/2} : H\to H$ is bounded.
    \item There exists $L\geq0$ such that $\forall x,y\in H, p,p'\in H, \ \mu,\mu'\in\mathcal{P}_1(H)$,
\[
|\mathcal{H}(x,p,\mu)-\mathcal{H}(y,p',\mu')|\leq L\left (\|x-y\|_{-1}+\mathbf{d}_{1,-1}(\mu,\mu') \right )(1+\|p\|+\|p'\|)+L\|p-p'\|.
\]
Moreover, the function $\mathcal{H}(\cdot,0,\cdot)$ is bounded.
\item $\mathcal{H}$ is Fr\'echet differentiable with respect to the second variable, $\mathcal{H}_{p}$ is bounded on $H\times B_R(0)\times \mathcal{P}_1(H)$ for every $R>0$, and there exists $C>0$ such that
\[
\|\mathcal{H}_{p}(x,p,\mu)-\mathcal{H}_{p}(y,p',\mu')\|\leq C(\|x-y\|_{-1}+\|p-p'\|+ \mathbf{d}_{1,-1}(\mu,\mu')), \ \ \ \ \forall x,y\in H, \ p,p'\in H, \ \mu,\mu'\in\mathcal{P}_1(H).
\]
\item
Let $\Xi$ be a real, separable Hilbert space. The function $\sigma:H\to L_2(\Xi,H)$ is bounded and there is a constant $C\geq 0$ such that
\begin{equation}
    \label{assumption2}
     \|\sigma(x)-\sigma(y)\|_{L_2(\Xi,H)} \leq C\|x-y\|_{-1}\quad\forall \,x,y\in H.
\end{equation}

\item
The function $G:H\times \mathcal{P}_1(H) \to \mathbb{R}$ is bounded, it is uniformly continuous in the $\|\cdot\|_{-1}\times \mathbf{d}_{1,-1}$ metric on bounded subsets of $H\times \mathcal{P}_1(H)$ and for every $\mu\in \mathcal{P}_1(H)$, $G(\cdot,\mu)$ extends to a function in $C^{1,1}_b(H_{-1})$. Moreover, there exists $\hat M\geq 0$ such that for every $\mu\in \mathcal{P}_1(H)$
\[
\|G(\cdot,\mu)\|_{C^{1,1}_b(H_{-1})}\leq \hat M.
\]

\item
Let $n\geq 2c_0, N\in \mathbb{N}$, let $m\in S$ and let $G$ satisfy (A5). Let $v,v_n, v_{n,N}$ be the viscosity solutions of \eqref{HJB}, \eqref{HJBn} and \eqref{HJBnN}, respectively (see comments below). For every $t\in[0,T]$ the functions $v(t,\cdot),v_n(t,\cdot), v_{n,N}(t,\cdot)$ extend to functions in $C^{1,1}_b(H_{-1})$, and there exists a constant $\hat M_1\geq 0$ such that
\[
\sup_{t\in[0,T]} \left(\|v(t,\cdot)\|_{C^{1,1}_b(H_{-1})}+\|v_n(t,\cdot)\|_{C^{1,1}_b(H_{-1})}+\|v_{n,N}(t,\cdot)\|_{C^{1,1}_b(H_{-1})}\right)\leq \hat M_1.
\]
\end{enumerate}
\end{assumption}

\begin{remark}\label{remark_convergence_Dv}
\begin{enumerate}[label=(\roman*)]
    \item We note that if Assumption \ref{ass:HJB}(A1) is satisfied, the Yosida approximations $A_n$ and the operators $A_{n,N}$ satisfy the weak $B$-condition with constant $2c_0$ for $n\geq 2c_0, N\in \mathbb{N}$, see \cite[page 264]{fabbri_gozzi_swiech_2017}.
    \item If Assumptions \ref{ass:HJB}(A1),(A2),(A4),(A5) are satisfied (in fact only Lipschitz continuity of $G(\cdot,\mu)$ with respect to the $H_{-1}$-norm is needed) then, by \cite[Theorem 3.50]{fabbri_gozzi_swiech_2017} and \cite[Theorem 3.86]{fabbri_gozzi_swiech_2017}, for every $m\in S$, equations \eqref{HJB}, \eqref{HJBn} and \eqref{HJBnN}, $n\geq 2c_0,N\geq 1$, have unique (within the class of bounded $B$-continuous functions) bounded viscosity solutions $v, v_n$ and $v_{n,N}$, respectively, which are bounded uniformly in $n,N$. We also have that $v_{n,N}\to v_n$ as $N\to\infty$, and $v_n\to v$ as $n\to \infty$ uniformly on bounded subsets of $[0,T]\times H$. Moreover, since $A_n, A_{n,N}$ are bounded, it follows from the construction of $v_n, v_{n,N}$, that they are also the viscosity solutions of \eqref{HJBn} and \eqref{HJBnN} in the usual sense where we use the standard $C^{1,2}((0,T)\times H)$ class of test functions. We also remark that the equations \eqref{HJB}, \eqref{HJBn} and \eqref{HJBnN} have comparison principles.
    \item Suppose that Assumptions \ref{ass:HJB}(A1),(A2),(A4),(A5),(A6) hold and let $m_1,m_2\in S$. Let $v_1$ and $v_2$ be the viscosity solutions of \eqref{HJB} associated with $m_1$ and $m_2$, respectively. Then $v_1^\pm(t,x)=v_1(t,x)\pm\|G(\cdot,m_1(T))-G(\cdot,m_2(T))\|_\infty \pm L(1+2\hat M)\rho_{-1,\infty}(m_1,m_2)(T-t)$ is a viscosity super/subsolution of \eqref{HJB} associated with $m_2$, and thus $v_1^-\leq v_2\leq v_1^+$. In particular, this implies that if $\rho_{-1,\infty}(m_k,m)\to 0$ as $k\to\infty$, then the corresponding viscosity solutions $v_k$ of \eqref{HJB} associated with $m_k$ converge uniformly to the viscosity solution of \eqref{HJB} associated with $m$. Following the argument of \cite[Lemma A.1]{defeo_swiech_wessels_2025}, it is then standard to observe that $Dv_k\to Dv$ uniformly on $[0,T]\times H$.
    \item If $A$ is a linear densely defined maximal dissipative operator in $H$ then $B:=((-A+I)(-A^*+I))^{-1}=(-A^*+I)^{-1}(-A+I)^{-1}$ always satisfies the weak $B$-condition with $c_0 =1$ since
    \begin{equation}
        -A^*B+B=(-A^*+I)B=(-A+I)^{-1}\geq 0.
    \end{equation}
    Moreover, it follows from e.g. \cite[Proposition B.2]{fabbri_gozzi_swiech_2017} that $\text{Range}(B^{1/2})=\mathcal{D}(A^*)$ and hence $A^*B^{1/2}$ is bounded.
\end{enumerate}
\end{remark}

The operator $B$ above is compact in many interesting cases. Let us discuss here a second order differential operator. Another example of a pair $A,B$ satisfying Assumption \ref{ass:HJB}(A1) is discussed in Section \ref{section_optimal_advertising}. There, the operator $A$ is a transport-type differential operator that arises in the context of delay differential equations.

\begin{example}
    For instance, let ${\mathcal O}\subset {\mathbb R}^n$ be a bounded domain with smooth boundary and let
    \begin{equation}
    \begin{cases}
        A:= \sum_{i,j=1}^n \partial_i (a_{ij}\partial_j) + \sum_{i=1}^n b_i\partial_i+c\\
        \mathcal{D}(A) := H^1_0(\mathcal{O}) \cap H^2(\mathcal{O}),
    \end{cases}
    \end{equation}
    where $a_{ij}=a_{ji}\in C^{1}(\overline{\mathcal{O}})$, $i,j\in \{1,\dots,n\}$, satisfy
    \begin{equation}
        \sum_{i,j=1}^n a_{ij}\xi_i\xi_j\geq \theta|\xi|^2\quad \forall \xi\in\mathbb{R}^n
    \end{equation}
    in $\mathcal{O}$ for some $\theta >0$, and $b_i,c \in L^\infty(\mathcal{O})$ for $i \in \{1,\dots,n\}$. Then, if $c \leq - C_0$ for some sufficiently large constant $C_0$ depending on $b_i$ and $a_{ij}$, $i,j\in \{1,\dots,n\}$, $A$ is maximal dissipative. Moreover, it is well known that $(-A+I)^{-1}: L^2({\mathcal O})\to H^2(\mathcal{O})$ is bounded and since the embedding of $H^1_0(\mathcal{O})$ into $L^2({\mathcal O})$ is compact, $(-A+I)^{-1}:L^2(\mathcal{O})\to L^2(\mathcal{O})$ is compact. Hence $B:=((-A+I)(-A^*+I))^{-1}$ is compact in $L(H)$, for $H=L^2(\mathcal{O})$.
\end{example}

\section{The Linear Fokker--Planck Equation: Weak Solutions}\label{sec:LFP}

In this section, we study a linear Fokker--Planck equation
\begin{equation}\label{FP_linear}
    \partial_t m(t) - L^* m(t) - \operatorname{div} \left (w(t,x) m(t) \right ) = 0,\quad m(0) = m_0.
\end{equation}
%where $L^*$ is the \textcolor{red}{formal} adjoint of
%\begin{equation}
%    L\phi(x) = \langle Ax, D\phi(x) \rangle + \frac12 \operatorname{Tr}[ \sigma(x) \sigma^*(x) D^2 \phi(x)].
%\end{equation}
The precise meaning of \eqref{FP_linear} is explained in Definition \ref{definition_weak_solution} below. We note that this equation arises from \eqref{FP} by setting $w(t,x) = \mathcal{H}_p(x,Dv(t,x),m(t))$. We make the following assumption.

\begin{assumption}\label{Assumption:1}
The function $w:[0,T]\times H\to H$ is continuous and bounded; furthermore, there is a constant $C_1 \geq 0$ such that
\begin{equation}
    \label{assumption1} \|w(t,x)-w(t,y)\| \leq C_1\|x-y\|_{-1}\quad\forall \,t\in [0,T],x,y\in H.
\end{equation}
\end{assumption}

We introduce the class of test functions
\[
    \mathcal{D}_T = \Big\{ \varphi \in C_{b}^{1,2}([0,T]\times H) : A^* D\varphi(t,x) \in C_b([0,T]\times H;H)\Big\}.
\]
Let $L_0$ be the operator defined on functions $\varphi \in \mathcal{D}_T$ by
\begin{equation}
    (L_0\varphi)(x) = \langle x,A^*D\varphi(x) \rangle + \frac12 \operatorname{Tr} [ \sigma(x) \sigma^*(x) D^2 \varphi(x) ].
\end{equation}

We denote by $S_{FP}$ the set of finite Borel measures $m$ on $[0,T]\times H$ of the form $m(\mathrm{d}t,\mathrm{d}x)=m(t,\mathrm{d}x)\mathrm{d}t$ such that $\int_0^T\int_H\|x\|m(t,\mathrm{d}x)\mathrm{d}t<\infty$. We recall that both $S$ and $\hat S$ are contained in $S_{FP}$. Moreover, every measure in $S_{FP}$ can be naturally extended to a Borel measure on $[0,T]\times H_{-1}$ by setting $m(F)=m(F\cap [0,T]\times H)$ for $F\in {\mathcal B}([0,T]\times H_{-1})$.

\begin{definition}\label{definition_weak_solution}
    A measure $m\in S_{FP}$ is a weak solution of equation \eqref{FP_linear} with initial condition $m_0 \in \mathcal{P}_1(H)$ if
    \begin{equation}
    \begin{split}
        &\int_H \varphi(t,x) m(t,\mathrm{d}x) - \int_H \varphi(0,x) m_0(\mathrm{d}x)\\
        &= \int_0^t \left ( \int_H \left [ \partial_t \varphi(s,x) + L_0 \varphi(s,x) - \langle w(s,x), D\varphi(s,x) \rangle \right ] m(s,\mathrm{d}x) \right ) \mathrm{d}s
    \end{split}
    \end{equation}
    for every $t\in [0,T]$ and $\varphi \in \mathcal{D}_T$.
\end{definition}

We note that if 
$\varphi\in \mathcal{D}_T$, then the function $f(s,x):=\partial_t \varphi(s,x) + L_0 \varphi(s,x) - \langle w(s,x), D\varphi(s,x)\rangle$ is continuous on $[0,T]\times H$ and satisfies the growth condition $|f(s,x)|\leq C(1+\|x\|)$. Hence $f\in L^1([0,T]\times H;m)$ and the integrals above are well-defined.

%Hence $g(s):=\int_H  f(s,x)m(s)(\mathrm{d}x)$ is well defined if $m\in S$. We will show that $g$ is Borel measurable.

%We first approximate $f$ by functions $f_m\in UC_b([0,T]\times H)$ such that $f_m\to f$ pointwise and also $|f_m(s,x)|\leq C(1+\|x\|)$ for $m=1,2,\dots.$ Then we %define for $n=1,2,\dots$ the functions $f_{m,n}(t,x):=f_m(t,P_nx)$. We have $f_{m,n}\in UC_b([0,T]\times H_{-1})$. Since convergence in $d_{1,-1}$ implies weak %convergence, it is easy to see that the functions $g_{m,n}(s):=\int_H f_{m,n}(s,x)m(s)(\mathrm{d}x)$ are continuous on $[0,T]$ and, by dominated convergence theorem, %converge to $g_{m}(s):=\int_H f_{m}(s,x)m(s)(\mathrm{d}x)$ for every $s\in[0,T]$. Hence the functions $g_m$ are Borel measurable. Applying dominated convergence %theorem again we see that $g_{m}(s)$ converges for every $s$ to $g(s)$. Hence $g$ is Borel measurable.

\subsection{Uniqueness}

We prove the uniqueness of a weak solution via a variation of the classical method based on establishing the density of the range of the operator 
$\partial_t \varphi + L_0 \varphi - \langle w, D\varphi \rangle$. To this end, let us introduce the equation
\begin{equation}\label{equation_1}
    \partial_t u +  \langle Ax,Du \rangle +\frac12 \operatorname{Tr} [ \sigma(x) \sigma^*(x) D^2 u ]  - \langle w(t,x), Du \rangle = f(t,x),\quad u(T,\cdot) = 0.
\end{equation}
Due to the unbounded operator and the infinite dimensionality of the underlying space, we consider two approximating problems. For $n,N\geq 1$, we introduce the equations
\begin{equation}\label{equation_1n}
    \partial_t u + \langle A_nx,Du \rangle +\frac12 \operatorname{Tr} [ \sigma(x) \sigma^*(x) D^2 u ]  - \langle w(t,x), Du \rangle = f(t,x),\quad u(T,\cdot) = 0,
\end{equation}
and
\begin{equation}\label{equation_1nN}
    \partial_t u + \langle A_{n,N}x,Du \rangle + \frac12 \operatorname{Tr} [ \sigma(P_Nx) \sigma^*(P_Nx)P_N D^2 u P_N ]  - \langle P_Nw(t,P_Nx), Du \rangle = f(t,P_Nx),\quad u(T,\cdot) = 0.
\end{equation}

We first collect a few known results about these equations.

\begin{lemma}\label{lem:KolmogorovLip}
Let Assumptions \ref{ass:HJB}(A1),(A4) and Assumption \ref{Assumption:1} be satisfied.
Suppose that $f\in C_b([0,T]\times H)$ is such that there is $C_2\geq 0$ such that for all $t\in[0,T], x,y\in H$,
\begin{equation}\label{eq:KLipwf}
|f(t,x) - f(t,y)|\leq C_2\|x-y\|_{-1}.
\end{equation}
Then, there exists a unique (in the class of bounded $B$-continuous functions) bounded viscosity solution $u$ of equation \eqref{equation_1} and for every $n\geq 2c_0$ there exist unique bounded viscosity solutions $u_n$ of equation \eqref{equation_1n} and $u_{n,N}$ of equation \eqref{equation_1nN}. We have 
$\|u\|_\infty,\|u_n\|_\infty, \|u_{n,N}\|_\infty\leq T\|f\|_\infty$. Moreover, there exists $M\geq 0$ such that for all $R>0$, $t,s\in[0,T], x,y\in H,\|x\|,\|y\|\leq R$ and $n,N$ as above,
\begin{equation}\label{eq:KLip1}
|u(t,x) - u(s,y) | +|u_n(t,x) - u_n(s,y) |+ |u_{n,N}(t,x) - u_{n,N}(s,y) |\leq M\|x-y\|_{-1}+K_R |t-s|^{\frac{1}{2}}
\end{equation}
The constant $M$ only depends on $T,\|B\|_{L(H)},C_1,C_2$, and the constant $C$ from Assumption \ref{ass:HJB}(A4); the constants $K_R$ may also depend on $\|A^*B\|_{L(H)},\|f\|_\infty, \|w\|_\infty, \|\sigma\|_\infty$. Moreover the functions $u_{n,N}$ converge to $u_n$ as $N\to \infty$, and 
$u_{n}$ converge to $u$ as $n\to \infty$, uniformly on bounded subsets of $[0,T]\times H$.
\end{lemma}
\begin{proof}
The existence of unique viscosity solutions to these equations follows from \cite[Theorem 3.66]{fabbri_gozzi_swiech_2017}. Estimate \eqref{eq:KLip1} is standard and follows easily from the stochastic representation formulas for $u,u_n$ and $u_{n,N}$, see \cite[Theorem 3.66]{fabbri_gozzi_swiech_2017} and \cite[Lemmas 3.20, 3.22]{fabbri_gozzi_swiech_2017}. Regarding convergence of $u_n$, since bounded and closed subsets of $[0,T]\times H$ are compact in $[0,T]\times H_{-1}$, using Arzel\`{a}--Ascoli Theorem we obtain that every subsequence of the sequence $(u_n)$ has a further subsequence which converges uniformly on bounded subsets of $[0,T]\times H$ to some function $\bar u$ which is bounded and satisfies \eqref{eq:KLip1}. It then follows from \cite[Theorem 3.41]{fabbri_gozzi_swiech_2017} that $\bar u$ is a bounded viscosity solution of \eqref{equation_1} and, since this equation has a unique viscosity solution, $\bar u$ must be equal to $u$. This implies convergence of the whole sequence $(u_n)$ to $u$. Similar arguments give the required convergence of $(u_{n,N})$ to $u_n$. We remark that the required convergence of $u_n$ to $u$ can also be easily deduced from the stochastic representation formulas for $u_n$ and $u$ since we have
\[
\lim_{n\to\infty}\sup\{\|\mathrm{e}^{tA}x-\mathrm{e}^{tA_n}x\|_{-1}:t\in[0,T],x\in H,\|x\|\leq 1\}=0.
\]
The latter is an easy consequence of compactness of $\{\|x\|\leq 1\}$ in $H_{-1}$ and the fact that $\|\mathrm{e}^{tA}x\|_{-1}+\|\mathrm{e}^{tA_n}x\|_{-1}\leq C(T)\|x\|_{-1}
$ for all $x\in H$, which follows from \cite[Lemma 3.19(ii)]{fabbri_gozzi_swiech_2017}.
\end{proof}
\begin{remark}\label{rem:Defvs}
\begin{enumerate}[label=(\roman*)]
    \item We note that the solutions $u,u_n$ and $u_{n,N}$ are the value functions for the associated SDE (there is no control here), see \cite[Theorem 3.66]{fabbri_gozzi_swiech_2017}. Since $A_n, A_{n,N}$ are bounded, it also follows that the value functions, and hence $u_n, u_{n,N}$, are the viscosity solutions of \eqref{equation_1n} and \eqref{equation_1nN} in the usual sense where we use the standard $C^{1,2}((0,T)\times H)$ class of test functions. Hence, the two definitions give the same viscosity solution for \eqref{equation_1n} and \eqref{equation_1nN}.
    \item Due to \eqref{eq:KLip1}, the functions $u,u_n,u_{n,N}$ extend to functions on $[0,T]\times H_{-1}$. 
\end{enumerate}
\end{remark}

\begin{lemma}\label{lem:Kolmogorov.C11}
Let the assumptions of Lemma \ref{lem:KolmogorovLip} be satisfied and suppose in addition there is $C_3\geq 0$ such that for all $t\in [0,T]$ and $x,y\in H$
\begin{equation}\label{eq:KLipa2}
        \| (D\sigma(x) - D\sigma(y))(x-y) \|_{L_2(\Xi,H_{-1})} \leq C_3 \|x-y\|_{-1}^2,
\end{equation}
\begin{equation}\label{eq:KLipa3}
\|(Dw(t,x)-Dw(t,y))(x-y)\|_{-1}\leq C_3\|x-y\|_{-1}^2,\quad \|Df(t,x)-Df(t,y)\|_{-1}\leq C_3\|x-y\|_{-1}.
\end{equation}
Then the viscosity solutions $u, u_n,u_{n,N}$ from Lemma \ref{lem:KolmogorovLip} satisfy
\begin{equation}\label{eq:KC11}
    \max_{t\in[0,T]}\left(\|u(t,\cdot)\|_{C^{1,1}_b(H_{-1})}+\|u_n(t,\cdot)\|_{C^{1,1}_b(H_{-1})} + \|u_{n,N}(t,\cdot)\|_{C^{1,1}_b(H_{-1})} \right) \leq L
\end{equation}
for some absolute constant $L\geq 0$ which depends on the constants from Lemma \ref{lem:KolmogorovLip} and $C_3$.
\end{lemma}

\begin{proof}
   Estimate \eqref{eq:KC11} can be proved using the fact that a uniformly continuous bounded function from $H_{-1}$ to $\mathbb{R}$ that is semiconvex and semiconcave, is $C^{1,1}_b(H_{-1})$, see \cite{lasry_lions_1986}. Here, the semiconvexity and semiconcavity follows from the Feynman--Kac representation of $u$ (resp. $u_n$ and $u_{n,N}$) and the semiconvexity and semiconcavity of $f$. The details can be found in \cite[Theorem 5.10]{defeo_swiech_wessels_2025}. Note that in that reference, the coefficients did not depend explicitly on $t$; instead they depend on a control variable. Nevertheless, the proof of the semiconcavity estimate for $u(t,\cdot), u_n(t,\cdot),u_{n,N}(t,\cdot)$ here follows exactly the same arguments. Since there is no control in the present case, the same proof gives uniform semiconvexity of $u(t,\cdot), u_n(t,\cdot), u_{n,N}(t,\cdot)$.
 \end{proof}

\begin{remark}\label{rem:Lipt}
It is standard to notice, see \cite[Remark 4.9]{swiech_wessels_2025}, that under the assumptions of Lemma \ref{lem:Kolmogorov.C11}, for every $R>0$ there is a constant $L_R\geq 0$ such that
\begin{equation}\label{eq:KLipt}
|u_{n,N}(t,x) - u_{n,N}(s,x) |,|u_n(t,x) - u_n(s,x) |\leq L_R|t-s|\quad\text{for all}\,\,t,s\in[0,T],x\in H,\|x\|\leq R.
\end{equation}
Moreover, see e.g. \cite[Lemma A.1]{defeo_swiech_wessels_2025}, $Du, Du_n, Du_{n,N}$ are uniformly continuous on bounded subsets of $[0,T]\times H$. The same arguments as those used to prove \cite[Lemma A.1]{defeo_swiech_wessels_2025} also show that $Du_n\to Du$ and $Du_{n,N}\to Du_n$ uniformly on bounded subsets of $[0,T]\times H$. We also notice that \eqref{eq:KC11} implies that for every $t\in [0,T]$, $D_{H_{-1}}u_{n,N}(t,\cdot):H_{-1}\to H_{-1}$ is Lipschitz continuous with a uniform Lipschitz constant independent of $n,N$. Therefore, since $Du_{n,N}=BD_{H_{-1}}u_{n,N}$ and $A^*B^{1/2}\in L(H)$, we obtain that $A^*Du_{n,N}(t,\cdot)=A^*BD_{H_{-1}}u_{n,N}(t,\cdot)=A^*B^{1/2}B^{1/2}D_{H_{-1}}u_{n,N}(t,\cdot)$ is Lipschitz continuous as a map from $H_{-1}$ to $H$ with a uniform Lipschitz constant independent of $n,N$ and $\|A^*Du_{n,N}\|_\infty\leq \|A^*B^{1/2}\|_{L(H)}M$. Therefore the same is true for $A^*_{n,N}Du_{n,N}$.
\end{remark}

The following lemma is the main ingredient in the proof of uniqueness for the linear Fokker--Planck equation \eqref{FP_linear}.

\begin{lemma}\label{lem:Kolmogorov.approx}
Let Assumptions \ref{ass:HJB}(A1),(A4) and Assumption \ref{Assumption:1} be satisfied and let $f\in C_b([0,T]\times H)$ satisfy \eqref{eq:KLipwf}. Suppose that 
\eqref{eq:KLipa2} and \eqref{eq:KLipa3} hold. Let $u$ be the viscosity solution of \eqref{equation_1}. Then there exist functions $\hat{u}_n\in \mathcal{D}_T$, $n=1,2,\dots$, such that $\hat{u}_n\to u, D\hat{u}_n \to Du$ uniformly on bounded subsets of $[0,T]\times H$ and such that 
\begin{equation}\label{equation_111}
    \partial_t \hat{u}_n(t,x) + L_0 \hat{u}_n(t,x) - \langle w(t,x), D\hat{u}_n(t,x) \rangle = f_n(t,x) \quad\text{in}\,\,(0,T)\times H,
\end{equation}
where $f_n\to f$ pointwise and $|f_n(t,x)|\leq C(1+\|x\|)$ for some $C\geq 0$ independent of $n$. Moreover, 
\begin{equation}\label{Bound_for_DH-1uhat}
\|D_{H_{-1}}\hat{u}_n\|_{-1,\infty}\leq M,
\end{equation}
where $M$ is the constant from \eqref{eq:KLip1}.
\end{lemma}

\begin{proof}
The proof is based on three approximation procedures: (i) a Yosida approximation $A_n$, $n\in \mathbb{N}$, of the unbounded operator $A$; (ii) a finite dimensional projection $P_N$, $N\in \mathbb{N}$, in the Hilbert space $H$; and (iii) a mollification which introduces a third parameter $\eta>0$.

First, let $u_n$, $n\geq 2c_0$, and $u_{n,N}$, $N\geq 1$ be the viscosity solutions of \eqref{equation_1n} and \eqref{equation_1nN}, respectively. By Lemmas \ref{lem:KolmogorovLip}, \ref{lem:Kolmogorov.C11} and Remark \ref{rem:Lipt} the functions $u_n, u_{n,N}$ satisfy \eqref{eq:KLip1}, \eqref{eq:KC11}, \eqref{eq:KLipt}. We note that the functions $u_{n,N}$ are obtained by first finding the viscosity solutions $\tilde{u}_{n,N}:[0,T]\times H_N\to\mathbb{R}$ of the equation
  \begin{equation}\label{equation_1nNfd}
    \begin{cases}
    \partial_t \tilde{u}_{n,N} +\langle A_{n,N}x,D\tilde{u}_{n,N} \rangle + \frac12 \operatorname{Tr} [ \sigma(x) \sigma^*(x) P_N D^2 \tilde{u}_{n,N} P_N ] \\
  \qquad\qquad\qquad - \langle w(t,x), D\tilde{u}_{n,N}\rangle = f(t,x)\quad\text{in } [0,T]\times H_N\\
    \tilde{u}_{n,N}(T,\cdot) = 0,
    \end{cases}
\end{equation}
and then setting $u_{n,N}(t,x)=\tilde{u}_{n,N}(t,P_Nx)$. Moreover, $\tilde{u}_{n,N}$ also satisfy \eqref{eq:KLip1}, \eqref{eq:KC11}, \eqref{eq:KLipt}. We point out that above we consider $H_N$ as a subspace of $H$ and hence consider $D\tilde{u}_{n,N}$ as an element of $H$ and $D^2\tilde{u}_{n,N}$ as an element of $S(H)$ using the identification discussed in Section \ref{sec:notass}. Moreover, we identify $x= \sum_{i=1}^{N} x^i e_i\in H_N$ with $(x^1,\dots,x^N)\in {\mathbb{R}}^N$ and hence we identify $H_N$ with ${\mathbb{R}}^N$. With this identification we can write $\tilde{u}_{n,N}(t,x)=\tilde{u}_{n,N}(t,x^1,\dots,x^N)$. Then $D\tilde{u}_{n,N}(t,x) = \sum_{i=1}^{N} \partial_{x^i}\tilde{u}_{n,N}(t,x) e_i$ and $D^2\tilde{u}_{n,N}(t,x)$ can be identified with the Hessian matrix $(\partial^2_{x^ix^j}\tilde{u}_{n,N}(t,x))_{1\leq i,j\leq N}$ by $D^2\tilde{u}_{n,N}(t,x)z=\sum_{i=1}^N\sum_{j=1}^N\partial^2_{x^ix^j}\tilde{u}_{n,N}(t,x)z^j e_i$ for $z=\sum_{i=1}^{N} z^i e_i$.

In order to use a mollification, we extend the functions $\tilde{u}_{n,N}$ from $[0,T]\times\mathbb{R}^N$ to ${\mathbb{R}}^{N+1}$ by setting $\tilde{u}_{n,N}(t,x)=\tilde{u}_{n,N}(T,x)=0$ for $t>T$ and $\tilde{u}_{n,N}(t,x)=\tilde{u}_{n,N}(0,x)$ for $t<0$. Now, let $\rho \in C^{\infty}({\mathbb{R}}^{N+1})$ be a standard mollifier with $\text{supp}(\rho) \subset B_1(0)$, $\rho(t,x) \geq 0$, $\int_{{\mathbb{R}}^{N+1}} \rho(t,x) \mathrm{d}t\mathrm{d}x = 1$, and let for $\eta>0$, $\rho_\eta(t,x) = (1/\eta^{N+1}) \rho((t,x)/\eta)$. We define
    \begin{equation}
       \tilde{u}_{n,N}^{\eta}(t,x) =\int_{\mathbb{R}^{N+1}} \tilde{u}_{n,N}(s,y) \rho_{\eta}(t-s,x-y)  \mathrm{d}s\mathrm{d}y \quad \text{for } t\in [0,T], x\in \mathbb{R}^N
    \end{equation}
    and then
    \[
    u_{n,N}^{\eta}(t,x)=\tilde{u}_{n,N}^{\eta}(t,P_Nx)\quad\text{for}\,\,t\in[0,T],x\in H.
    \]
Let us collect a few properties of these mollified functions. We have $\tilde{u}_{n,N}^{\eta}\in C^{\infty}([0,T]\times H_N)$ and hence $u_{n,N}^{\eta}\in C^{\infty}([0,T]\times H)$. 
%\textcolor{red}{$A^*_{n,N}Du_{n,N}^{\eta}\in C^{\infty}([0,T]\times H;\textcolor{blue}{H})$}. 
Using standard convolution arguments it is easy to see that 
\begin{equation}\label{eq:convvnNeta}
    \tilde{u}_{n,N}= \lim_{\eta\to 0} \tilde{u}_{n,N}^{\eta}, \quad D\tilde{u}_{n,N}= \lim_{\eta\to 0} D\tilde{u}_{n,N}^{\eta}, 
    %\quad \textcolor{red}{A^*_{n,N}D\tilde{u}_{n,N}= \lim_{\eta\to 0} A^*_{n,N}D\tilde{u}_{n,N}^{\eta}}
\end{equation}
uniformly on bounded subsets of $[0,T]\times H_N$ and hence
\begin{equation}\label{eq:convunNeta}
     u_{n,N}= \lim_{\eta\to 0} u_{n,N}^{\eta}, \quad Du_{n,N}= \lim_{\eta\to 0} Du_{n,N}^{\eta}, 
     %\quad \textcolor{red}{A^*_{n,N}Du_{n,N}= \lim_{\eta\to 0} A^*_{n,N}Du_{n,N}^{\eta}}
\end{equation}
uniformly on bounded subsets of $[0,T]\times H$. We also see that if $M$ is the constant from \eqref{eq:KLip1} and $L_R$, $R>0$, are constants from \eqref{eq:KLipt}, then for $\eta$ sufficiently small
\begin{equation}\label{eq:KLiptetak}
    |\tilde{u}_{n,N}^{\eta}(t,x) - \tilde{u}_{n,N}^{\eta}(s,y) |\leq L_{R+1}|t-s| +M\|x-y\|_{-1}\quad\text{for all}\,\,t,s\in[0,T],x,y\in H_N,\|x\|,\|y\|\leq R
\end{equation}
and thus the same is true for $u_{n,N}^{\eta}$. Hence $u_{n,N}^{\eta}$ extends to a function on $[0,T]\times H_{-1}$ and $\|D_{H_{-1}}u_{n,N}^{\eta}\|_{-1,\infty} \leq M$. Similarly we have $\|D_{H_{-1}}\tilde{u}_{n,N}^{\eta}\|_{-1,\infty} \leq M$.
%Moreover, since 
%\[
%\partial_tu_n^{\eta,k}(t,x)=\int_{\mathbb{R}^{k+1}} u_n(s,y^1,\dots,y^k,x') \rho_{\eta}'(t-s)\prod_{i=1}^k \rho_{\eta_i}(x^i-y^i)  \mathrm{d}%s\mathrm{d}y^1\dots\mathrm{d}y^k,
%\]
%which is Lipschitz continuous on $[0,T]\times H$, uniformly in $k$, and the sequence $(\partial_tu_n^{\eta,k})$ is uniformly convergent on $%[0,T]\times H$ as $k\to\infty$, we obtain that its limit is Lipschitz continuous on $[0,T]\times H$ and must be equal to $\partial_tu_n^{\eta}$. %Similarly, we have
%\[
%Du_n^{\eta,k}(t,x)=\int_{\mathbb{R}^{k+1}} Du_n(s,x^1-y^1,\dots,x_k-y^k,x') \rho_{\eta}(t-s)\prod_{i=1}^k \rho_{\eta_i}(y^i)  \mathrm{d}%s\mathrm{d}y^1\dots\mathrm{d}y^k,
%%\]
%from which, using \eqref{eq:KLipa3}, we easily deduce that the sequence $(Du_n^{\eta,k})$ is uniformly convergent on $[0,T]\times H$ as %$k\to\infty$ and
% \begin{equation}\label{eq:KLipxetak1}
%\|Du_n^{\eta,k}(t,x) - Du_n^{\eta,k}(t,y) \|\leq \|B^{\frac{1}{2}}\|L\|x-y\|_{-1}\quad\mbox{for all}\,\,t\in[0,T],x,y\in H,
%\end{equation}
%where $L$ is the constant from \eqref{eq:KC11}. We can then conclude that $u_n^\eta$ has Fr\'echet derivative with respect to $x$ which %is equal to the limit of the sequence $(Du_n^{\eta,k})$. 
Moreover, \eqref{eq:KC11} implies that for all $t\in [0,T], x,y\in H_N$ and $0\leq \lambda\leq 1$,
\begin{equation}\label{eq:C11vnNeta}
\begin{split}
&|\lambda \tilde{u}_{n,N}^{\eta}(t,x) +(1-\lambda)\tilde{u}_{n,N}^{\eta}(t,y) - \tilde{u}_{n,N}^{\eta}(t,\lambda x+(1-\lambda)y)|
\\
&\leq
\int_{\mathbb{R}^{N+1}} \Big |\lambda \tilde{u}_{n,N}(s,x-z) +(1-\lambda)\tilde{u}_{n,N}(s,y-z)
\\
&
\qquad\qquad
-\tilde{u}_{n,N}(s, \lambda (x-z)+(1-\lambda) (y-z))\Big |
\rho_{\eta}(t-s,z)  \mathrm{d}s\mathrm{d}z
\\
&
\leq 
\int_{\mathbb{R}^{N+1}} \lambda(1-\lambda)L\|x-y\|_{-1}^2\rho_{\eta}(t-s,z)  \mathrm{d}s\mathrm{d}z= \lambda(1-\lambda)L\|x-y\|_{-1}^2.
\end{split}
\end{equation}
%This, together with \eqref{eq:KLiptetak}, also  implies that \textcolor{blue}{$\|D_{H_{-1}}\tilde{u}_{n,N}^{\eta}\|_{-1,\infty} \leq M$} and hence 
%\textcolor{blue}{$\|D_{H_{-1}}u_{n,N}^{\eta}\|_{-1,\infty} \leq M$}. 
Thus we conclude that the functions 
%\textcolor{blue}{$\tilde{u}_{n,N}^{\eta}$} and 
$u_{n,N}^{\eta}$ 
satisfy \eqref{eq:KC11} with some constant $L$ independent of $n,N,\eta$. Also, arguing as in Remark \ref{rem:Lipt}, we have $\|A^* Du_{n,N}^{\eta}\|_\infty \leq \|A^*B^{1/2}\|_{L(H)}M$ and $\|A^*_{n,N}D u_{n,N}^{\eta}\|_\infty \leq \|A^*B^{1/2}\|_{L(H)}M$. 
%\textcolor{blue}{and the same is true for $u^{\eta}_{n,N}$.} 
This shows that $u^{\eta}_{n,N} \in \mathcal{D}_T$ for any $\eta>0$ and $n,N\in \mathbb{N}$.

%\textcolor{blue}{Since the solution $\tilde{u}_{n,N}$ of equation \eqref{equation_1nNfd} satisfies \eqref{eq:KLip1}, \eqref{eq:KC11},%
%\eqref{eq:KLipt}}, 
Since $\tilde{u}_{n,N}\in W^{1,2,\infty}_{\rm loc}([0,T]\times H_N)$, it satisfies the equation pointwise a.e., that is for a.e. $(s,y)=(s,y^1,\dots,y^N)\in [0,T]\times H_N$
\[
\begin{split}
 \partial_t \tilde{u}_{n,N}(s,y) &+\langle y,A_{n,N}^*D\tilde{u}_{n,N}(s,y) \rangle + \frac12 \operatorname{Tr} [\sigma(y) \sigma^*(y)P_ND^2 \tilde{u}_{n,N}(s,y)P_N ] 
\\
&- \langle w(s,y), D\tilde{u}_{n,N}(s,y)\rangle = f(s,y).
\end{split}
\]
Let $(t,x)\in (0,T)\times H_N$. We multiply the above equation by $\rho_{\eta}(t-s,x-y)$, $\eta<t$, and integrate over $\mathbb{R}^{N+1}$. Using uniform continuity of $\sigma, f$, uniform continuity on bounded sets of $w$, uniform boundedness of $D\tilde{u}_{n,N}$, $A^*_{n,N}D\tilde{u}_{n,N}$, $D^2\tilde{u}_{n,N}$ and the linearity of the equation it is easy to see that the functions $\tilde{u}_{n,N}^\eta$ satisfy 
\begin{equation}\label{eq:equalityunNetatilde}
\begin{split}
 \partial_t \tilde{u}_{n,N}^\eta(t,x) &+\langle x,A_{n,N}^*D\tilde{u}_{n,N}^\eta(t,x) \rangle + \frac12 \operatorname{Tr} [ \sigma(x) \sigma^*(x)P_ND^2 \tilde{u}_{n,N}^\eta(t,x)P_N ] 
\\
&- \langle w(t,x), D\tilde{u}_{n,N}^\eta(t,x)\rangle = f^\eta_{n,N}(t,x),
\end{split}
\end{equation}
where $f^\eta_{n,N}:[\eta,T-\eta]\times H_N\to\mathbb{R}$, $\|f^\eta_{n,N} \|_\infty\leq \hat C$ for some $\hat C$ independent of $\eta$, $| f^\eta_{n,N} - f | \leq f^{\eta}$, and $f^{\eta} \to 0$ as $\eta \to 0$, uniformly on bounded subsets of $[\tau,T-\tau]\times \bigcup_{N\geq 1}H_N,\tau>0$. Indeed, regarding the estimate for the term involving $A_{n,N}^*$, we have
\[
\begin{split}
&\left|\int_{\mathbb{R}^{N+1}}\langle y,A_{n,N}^*D\tilde{u}_{n,N}(s,y) \rangle \rho_{\eta}(t-s,x-y)\mathrm{d}s\mathrm{d}y-\langle x,A_{n,N}^*D\tilde{u}_{n,N}^\eta(t,x) \rangle\right|\\
&\leq \eta \|A^*B^{\frac{1}{2}}\|_{L(H)}M+
\left|\left \langle x,\int_{\mathbb{R}^{N+1}}A_{n,N}^*D\tilde{u}_{n,N}(s,y) \rho_{\eta}(t-s,x-y)\mathrm{d}s\mathrm{d}y-A_{n,N}^*D\tilde{u}_{n,N}^\eta(t,x)\right \rangle \right |
\\
&=\eta \|A^*B^{\frac{1}{2}}\|_{L(H)}M+|\langle x,A_{n,N}^*D\tilde{u}_{n,N}^\eta(t,x)-A_{n,N}^*D\tilde{u}_{n,N}^\eta(t,x)\rangle|=\eta \|A^*B^{\frac{1}{2}}\|_{L(H)}M.
\end{split}
\]
The remaining terms can be handled similarly.

Hence, since $Du_{n,N}^\eta(t,x)=P_ND\tilde{u}_{n,N}^\eta(t,P_Nx)=D\tilde{u}_{n,N}^\eta(t,P_Nx)$ and then $D^2u_{n,N}^\eta(t,x)=P_ND^2\tilde{u}_{n,N}^\eta(t,P_Nx)P_N$, we obtain from \eqref{eq:equalityunNetatilde} 
\begin{equation}\label{eq:equalityunNeta}
\begin{split}
 \partial_t u_{n,N}^\eta(t,x) &+\langle x,A_{n,N}^*Du_{n,N}^\eta(t,x) \rangle + \frac12 \operatorname{Tr} [ \sigma(P_Nx) \sigma^*(P_Nx)D^2 u_{n,N}^\eta(t,x) ] 
\\
&- \langle w(t,P_Nx), Du_{n,N}^\eta(t,x)\rangle = f^\eta_{n,N}(t,P_Nx),\quad\text{for all}\,\,(t,x)\in[\eta,T-\eta]\times H.
\end{split}
\end{equation}
Due to Lemma \ref{lem:KolmogorovLip}, Remark \ref{rem:Lipt} and equation \eqref{eq:convunNeta}, for every $n\geq 2c_0$ we may choose $N(n)\geq n$ and $\eta(n)\leq 1/n$ such that for $t\in[0,T],\|x\|\leq n$
\[
\left | u_{n,N(n)}^{\eta(n)}(t,x)-u_n(t,x) \right | \leq \frac{1}{n},\quad \left \|Du_{n,N(n)}^{\eta(n)}(t,x)-Du_n(t,x) \right \|\leq \frac{1}{n}.
\]
We set $\hat{u}_n:=u_{n,N(n)}^{\eta(n)}$.

Next, we want to eliminate the projections $P_N$ in equation \eqref{eq:equalityunNeta}. To do so, let us treat each term individually. First we note that
\[
A_{n,N(n)}^*D\hat{u}_n(t,x)=P_{N(n)}n(nI-A^*)^{-1}A^*P_{N(n)}D\hat{u}_n(t,x)=P_{N(n)}n(nI-A^*)^{-1}A^*B^{\frac{1}{2}}P_{N(n)}B^{-\frac{1}{2}}D\hat{u}_n(t,x).
\]
Hence, noticing that $D\hat{u}_n(t,x)=BD_{H_{-1}}\hat{u}_n(t,x)$, we can write 
\[
\langle x,A_{n,N(n)}^*D\hat{u}_n(t,x) \rangle=\langle P_{N(n)}(A^*B^{\frac{1}{2}})^*n(nI-A)^{-1}P_{N(n)}x,B^{\frac{1}{2}}D_{H_{-1}}\hat{u}_n(t,x) \rangle.
\]
Since $\|n(nI-A)^{-1}\|_{L(H)}\leq 1$ and $n(nI-A)^{-1}y\to y$ for all $y\in H$, $n(nI-A)^{-1}P_{N(n)}x\to x$ as $n\to\infty$. Thus, as $A^*B^{1/2}\in L(H)$, we obtain $P_{N(n)}(A^*B^{1/2})^*n(nI-A)^{-1}P_{N(n)}x\to (A^*B^{1/2})^*x$ as $n\to\infty$. Therefore, since $\|B^{1/2}D_{H_{-1}}\hat{u}_n(t,x)\| = \|D_{H_{-1}}\hat{u}_n(t,x)\|_{-1}\leq M$, for every $(t,x)\in [0,T]\times H$, we have
\[
\begin{split}
g_1^n(t,x)&:=|\langle x,A_{n,N(n)}^*D\hat{u}_n(t,x) \rangle-\langle x,A^*D\hat{u}_n(t,x) \rangle|
\\
&
=
|\langle x,A_{n,N(n)}^*D\hat{u}_n(t,x) \rangle-\langle (A^*B^{\frac{1}{2}})^*x,B^{\frac{1}{2}}D_{H_{-1}} \hat{u}_n(t,x)\rangle|\to 0 \quad\text{as }n\to\infty
\end{split}
\]
and moreover $|g_1^n(t,x)|\leq M\|x\|$.

In order to treat the diffusion term in equation \eqref{eq:equalityunNeta}, we note that for every $(t,x)\in[0,T]\times H$, we have $D^2\hat{u}_n(t,x)=BD^2_{H_{-1}}\hat{u}_n(t,x)$, and thus $\|D^2\hat{u}_n(t,x)\|_{L(H)}\leq \|B\|_{L(H)}L$. Also, it follows from \eqref{assumption2} that
\[
\|\sigma(x)-\sigma(P_Nx)\|_{L_2(\Xi,H)}\to 0\quad \text{as } N\to \infty
\]
uniformly on bounded subsets of $H$. Hence,
\[
g_2^n(t,x) := \frac12\left|\operatorname{Tr} [ \sigma(P_{N(n)}x) \sigma^*(P_{N(n)}x)D^2 \hat{u}_n(t,x)] 
-\operatorname{Tr} [ \sigma(x) \sigma^*(x) D^2 \hat{u}_n(t,x)] \right| \to 0 \quad \text{as } n\to \infty
\]
uniformly on bounded subsets of $H$ and the functions $g_2^n$ are uniformly bounded.

Regarding the advection term in equation \eqref{eq:equalityunNeta}, since $B^{1/2}$ is compact, $\lim_{N\to+\infty}\|x-P_Nx\|_{-1}\to 0$ uniformly on bounded sets in $H$. This, together with the fact that
$\|D\hat{u}_n\|_\infty\leq \|B^{1/2}\|_{L(H)}M$ implies that
\[
g_3^n(t,x) := \left |\langle w(t,P_{N(n)}x), D\hat{u}_n(t,x)\rangle-\langle w(t,x), D\hat{u}_n(t,x)\rangle \right | \to 0\quad \text{as } n \to \infty
\]
uniformly on bounded subsets of $H$ and the functions $g_3^n$ are uniformly bounded.

Finally, regarding the source term in equation \eqref{eq:equalityunNeta} we note that
\begin{equation}
    g_4^n(t,x) := \left | f(t,x) - f^{\eta(n)}_{n,N(n)}(t, P_{N(n)}x) \right | \leq f^{\eta(n)} \to 0\quad \text{as } n \to \infty
\end{equation}
uniformly on bounded subsets of $H$ and the functions $g_4^n$ are uniformly bounded.

Applying these estimates to \eqref{eq:equalityunNeta} with $\hat{u}_n=u_{n,N(n)}^{\eta(n)}$ we thus conclude that
\begin{equation}\label{eq:equalityunNeta1}
\begin{split}
 \partial_t \hat{u}_n(t,x) &+\langle x, A^*D\hat{u}_n(t,x) \rangle + \frac12 \operatorname{Tr} [ \sigma(x) \sigma^*(x)D^2 \hat{u}_n(t,x) ] 
\\
&- \langle w(t,x), D\hat{u}_n(t,x)\rangle =f_n(t,x):=f(t,x)+ \hat g_n(t,x),\quad\text{for all}\,\,(t,x)\in[\eta(n),T-\eta(n)]\times H,
\end{split}
\end{equation}
where $|\hat g_n(t,x)|\leq g_1^n(t,x)+g_2^n(t,x)+g_3^n(t,x) + g_4^n(t,x)$, that is for every $(t,x)\in(0,T)\times H$, $|\hat g_n(t,x)|\leq C(1+\|x\|)$ and $\hat g_n(t,x)\to 0$ as $n\to\infty$. If $t\in (0,T)\setminus [\eta(n),T-\eta(n)]$ then we set $f_n(t,x)$ to be the left-hand side of \eqref{eq:equalityunNeta1} so that $|\hat f_n(t,x)|\leq C(1+\|x\|)$. However, if $0<t<T$ then for large $n$, $t\in [\eta(n),T-\eta(n)]$ and hence always for every $(t,x)\in (0,T)\times H$ we have $f_n(t,x)\to 0$.
\end{proof}

With a similar proof, we obtain the following lemma which will be used in Section \ref{section_MFG_Uniqueness} to prove uniqueness of the solution of the MFG system.

\begin{lemma}\label{lem:HJB.approx}
Let Assumptions \ref{ass:HJB}(A1),(A2),(A4),(A5),(A6) and let $m\in S$. Let $v$ be the viscosity solution of \eqref{HJB}. Then there exist functions $\hat{v}_n\in \mathcal{D}_T, n\geq 2c_0$ such that $\hat{v}_n\to v, D\hat{v}_n \to Dv$ uniformly on bounded subsets of $[0,T]\times H$ and such that 
\begin{equation}\label{equation_1HJB}
    \partial_t \hat{v}_n(t,x) + L_0 \hat{v}_n(t,x) - \mathcal{H}(x,D\hat{v}_n(t,x),m(t)) = f_n(t,x) \quad\text{in}\,\,(0,T)\times H,
\end{equation}
where $f_n(t,x)\to 0$ for all $(t,x)\in (0,T)\times H$ and $|f_n(t,x)|\leq C(1+\|x\|)$ for some $C\geq 0$ independent of $n$. 
\end{lemma}

\begin{proof}
We use the same notation and an analogous approximation procedure as in the proof of Lemma \ref{lem:Kolmogorov.approx}. Let $v_n$ and $v_{n,N}$ be the viscosity solutions of \eqref{HJBn} and \eqref{HJBnN}, respectively. Assumption \ref{ass:HJB}(A6), together with the argument from \cite[Remark 4.9]{swiech_wessels_2025} as before, guarantees that $v_n,$ and $v_{n,N}$ satisfy 
  \begin{equation}\label{eq:KLiptv}
|v_n(t,x) - v_n(s,y) |+ |v_{n,N}(t,x) - v_{n,N}(s,y) |\leq L\|x-y\|_{-1}+L_R |t-s|\quad\forall \,t,s\in[0,T],x,y\in H,\|x\|,\|y\|\leq R
\end{equation}
for all $R>0$, where $L$ is from Assumption \ref{ass:HJB}(A6) and the $L_R$ are some absolute constants. If we use \eqref{eq:KLiptv} and argue for instance as in the proof of \cite[Lemma A.1]{defeo_swiech_wessels_2025}, we easily obtain that $Dv,Dv_n,Dv_{n,N}$ are uniformly continuous on bounded subsets of $[0,T]\times H$, uniformly in $n,N$. The functions $v_{n,N}$ are obtained by first finding viscosity solutions $\tilde{v}_{n,N}:[0,T]\times H_N \to \mathbb{R}$ of the equation
    \begin{equation}\label{equation_1nNfd_2}
    \begin{cases}
        \partial_t \tilde{v}_{n,N} + \langle A_{n,N} x,D\tilde{v}_{n,N} \rangle + \frac12 \operatorname{Tr} [ \sigma(x) \sigma^*(x)P_N D^2 \tilde{v}_{n,N} P_N ] - \mathcal{H}(x,D\tilde{v}_{n,N},m(t)) = 0\quad \text{in } [0,T]\times H_N\\
        \tilde{v}_{n,N}(T,x) = G(x,m(T)), \quad x\in H_N  
    \end{cases}
    \end{equation}
    and then setting $v_{n,N}(t,x) = \tilde{v}_{n,N}(t,P_Nx)$. Moreover, $\tilde{v}_{n,N}$ also satisfies \eqref{eq:KC11}, \eqref{eq:KLiptv} (with $L_R$ replaced by $L_{R+1}$). Now, we use a mollification to define
    \begin{equation}
        \tilde{v}^{\eta}_{n,N}(t,x) = \int_{\mathbb{R}^{N+1}} \tilde{v}_{n,N}(s,y) \rho_{\eta}(t-s,x-y) \mathrm{d}s\mathrm{d}y \quad \text{for }(t,x)\in [0,T]\times H_N
    \end{equation}
    and set $v^{\eta}_{n,N} = \tilde{v}^{\eta}_{n,N}(t,P_Nx)$, for $(t,x) \in [0,T]\times H$. Note that since $v_n$, $v_{n,N}$ and $\tilde{v}_{n,N}$ satisfy \eqref{eq:KC11} and \eqref{eq:KLiptv}, $v^{\eta}_{n,N}$ and $\tilde{v}^{\eta}_{n,N}$ inherit the same properties as $u^{\eta}_{n,N}$ and $\tilde{u}^{\eta}_{n,N}$ in the proof of Lemma \ref{lem:Kolmogorov.approx}.  
    
    Let $(t,x)\in (0,T)\times H_N$. Since $\tilde{v}_{n,N}\in W^{1,2,\infty}_{\rm loc}([0,T]\times H_N)$, it satisfies equation \eqref{equation_1nNfd_2} pointwise a.e., that is for a.e. $(s,y)=(s,y^1,\dots,y^N)\in [0,T]\times H_N$
    \begin{equation}
        \partial_t \tilde{v}_{n,N}(s,y) + \langle y,A_{n,N}^*D\tilde{v}_{n,N}(s,y) \rangle + \frac12 \operatorname{Tr} [\sigma(y) \sigma^*(y)P_ND^2 \tilde{v}_{n,N}(s,y)P_N ]  - \mathcal{H}(y,D\tilde{v}_{n,N}(s,y),m(s)) = 0.
    \end{equation}
    Multiplying this equation by $\rho_{\eta}(t-s,x-y)$, $\eta<t$, and integrating yields
    \begin{multline}\label{equation_tilde_v}
         \partial_t \tilde{v}^{\eta}_{n,N}(t,x) +\langle y,A_{n,N}^*D\tilde{v}^{\eta}_{n,N}(t,x) \rangle\\
        + \frac12 \operatorname{Tr} [\sigma(x) \sigma^*(x)P_ND^2 \tilde{v}^{\eta}_{n,N}(t,x)P_N ]  - \mathcal{H}(x,D\tilde{v}^{\eta}_{n,N}(t,x),m(s)) = f^{\eta}_{n,N}(t,x)
    \end{multline}
 for some functions $f^{\eta}_{n,N}:[\eta,T-\eta]\times H_N \to \mathbb{R}$, satisfying $\| f_{n,N}^{\eta}\|_\infty \leq C$, where $C$ is independent of $\eta$, $| f_{n,N}^{\eta} - f | \leq f^{\eta}$, and $f^{\eta}\to 0$ as $\eta \to 0$ uniformly on bounded subsets of $[\tau,T-\tau]\times  \bigcup_{N\geq 1}H_N,\tau>0$. Indeed, the linear terms can be treated as in Lemma \ref{lem:Kolmogorov.approx}. For the nonlinear term, using the Lipschitz continuity of $\mathcal{H}$ from Assumption \ref{ass:HJB}(A2), we obtain
    \begin{equation}
    \begin{split}
        &\left | \int_{\mathbb{R}^{N+1}} \rho_{\eta}(t-s,x-y) \mathcal{H}(y,D\tilde{v}_{n,N}(s,y),m(s)) \mathrm{d}s\mathrm{d}y - \mathcal{H} \left ( x, D\tilde{v}^{\eta}_{n,N}(t,x), m(t) \right ) \right |\\
        %&\leq \int_{\mathbb{R}^{N+1}} \rho_{\eta}(t-s,x-y) \left | \mathcal{H}(y,D\tilde{v}_{n,N}(s,y),m(s)) - \mathcal{H} \left ( x, D\tilde{v}^{\eta}_{n,N}(t,x), m(t) \right ) \right |\mathrm{d}s\mathrm{d}y\\
        &\leq L \int_{\mathbb{R}^{N+1}} \rho_{\eta}(t-s,x-y) ( \| x-y \|_{-1} + \mathbf{d}_{1,-1}(m(s),m(t)) ) \left ( 1 + \|D\tilde{v}_{n,N}(s,y) \| + \left \| D\tilde{v}^{\eta}_{n,N}(t,x) \right \| \right ) \mathrm{d}s\mathrm{d}y\\
        &\quad + L \int_{\mathbb{R}^{N+1}} \rho_{\eta}(t-s,x-y) \left |D\tilde{v}_{n,N}(s,y) - D\tilde{v}^{\eta}_{n,N}(t,x) \right | \mathrm{d}s\mathrm{d}y.
    \end{split}
    \end{equation}
    Since $D\tilde{v}_{n,N}$ and $D\tilde{v}^{\eta}_{n,N}$ are uniformly bounded, $m\in S$ and $D\tilde{v}_{n,N}$ is uniformly continuous on bounded sets, uniformly in $n,N$, it is easy to see that both integrals above converge to $0$ as $\eta\to 0$ uniformly on bounded subsets of $[\tau,T-\tau]\times H_N,\tau>0$.
%For the second term on the right-hand side, we have
 %   \begin{equation}
  %  \begin{split}
   %     &\int_{\mathbb{R}^{N+1}} \rho_{\eta}(t-s,x-y) L \left |D\tilde{v}_{n,N}(s,y) - D\tilde{v}^{\eta}_{n,N}(t,x) \right | \mathrm{d}s\mathrm{d}y\\
   %     &= \int_{\mathbb{R}^{N+1}} \rho_{\eta}(t-s,x-y) L \left |D\tilde{v}_{n,N}(s,y) - \int_{\mathbb{R}^{N+1}} D\tilde{v}_{n,N}(r,z) \rho_{\eta}(t-r,x-z) \mathrm{d}r\mathrm{d}z \right | \mathrm{d}s\mathrm{d}y \\
    %    &= \int_{\mathbb{R}^{N+1}} \rho_{\eta}(t-s,x-y) L \int_{\mathbb{R}^{N+1}} \left | D\tilde{v}_{n,N}(s,y) -D\tilde{v}_{n,N}(r,z) \right | \rho_{\eta}(t-r,x-z) \mathrm{d}r\mathrm{d}z \mathrm{d}s\mathrm{d}y.
  %  \end{split}
  %  \end{equation}
   % Due to Remark \ref{rem:Lipt}, $D\tilde{v}_{n,N}$ is uniformly continuous on bounded sets. Hence, the right-hand side converges to $0$ as $
   %\eta \to 0$.
    
    Thus, since $Dv_{n,N}^\eta(t,x)=D\tilde{v}_{n,N}^\eta(t,P_Nx)$ and then $D^2v_{n,N}^\eta(t,x)=P_ND^2\tilde{v}_{n,N}^\eta(t,P_Nx)P_N$ we derive from \eqref{equation_tilde_v}
    \begin{multline}\label{equation_v_eta_projection}
         \partial_t v^{\eta}_{n,N}(t,x) + \langle x,A_{n,N}^*Dv^{\eta}_{n,N}(t,x) \rangle + \frac12 \operatorname{Tr} [\sigma(P_Nx) \sigma^*(P_Nx)P_ND^2 v^{\eta}_{n,N}(t,x)P_N ] \\
        - \mathcal{H}(P_Nx,Dv^{\eta}_{n,N}(t,x),m(t)) = f^{\eta}_{n,N}(t,P_N x), \quad \text{for all }(t,x)\in [\eta,T-\eta] \times H.
    \end{multline}
    We choose $N(n) \geq n$ and $\eta(n) \leq 1/n$ such that for $t\in [0,T]$, $\|x\| \leq n$
    \[
        |v_{n,N(n)}^{\eta(n)}(t,x)-v_n(t,x)|\leq \frac{1}{n},\quad \|Dv_{n,N(n)}^{\eta(n)}(t,x)-Dv_n(t,x)\|\leq \frac{1}{n}
    \]
    and set $\hat{v}_n:=v_{n,N(n)}^{\eta(n)}$. Next, we want to eliminate the projections in \eqref{equation_v_eta_projection}. The error terms corresponding to the linear terms can be estimated as in the proof of Lemma \ref{lem:Kolmogorov.approx}. For the nonlinear term, we observe that
    \begin{equation}
        \left | \mathcal{H}(P_Nx,D\hat{v}_{n}(t,x),m(t)) - \mathcal{H}(x,D\hat{v}_n(t,x),m(t)) \right | \leq L \| P_N x - x \|_{-1} ( 1 + 2 \| D\hat{v}_{n}(t,x) \| ) 
    \end{equation}
    Thus, using that $\hat{v}_n = v^{\eta(n)}_{n,N(n)}$, we obtain from \eqref{equation_v_eta_projection} \begin{multline}\label{eq:equalityunNeta1_2}
        \partial_t \hat{v}_n(t,x) +\langle x, A^*D\hat{v}_n(t,x) \rangle + \frac12 \operatorname{Tr} [ \sigma(x) \sigma^*(x)D^2 \hat{v}_n(t,x) ] \\
        - \mathcal{H}(x,D\hat{v}_n(t,x),m(t)) =f_n(t,x),\quad\text{for all}\,\,(t,x)\in[\eta(n),T-\eta(n)]\times H,
    \end{multline}
    where $f_n:(0,T)\times H \to \mathbb{R}$ satisfies $|f_n(t,x)| \leq C(1+\|x\|)$ and $f_n(t,x) \to 0$ as $n\to \infty$. Now we conclude as in the proof of Lemma \ref{lem:Kolmogorov.approx}. We remark that a separate bound on $\|D_{H_{-1}}\hat{v}_n\|_{-1,\infty}$, independent of the $C^{1,1}_b(H_{-1})$ bounds in Assumption \ref{ass:HJB}(A6), is not needed.
\end{proof}

Now, let us state and prove the main result of this section.

\begin{theorem}\label{th:FPuniqueness}
Let Assumptions \ref{ass:HJB}(A1),(A4) be satisfied and let $w$ satisfy Assumption \ref{Assumption:1}. Suppose that \eqref{eq:KLipa2} holds. Then for every $m_0\in {\mathcal P}_1(H)$ there exists at most one weak solution of the
linear Fokker--Planck equation \eqref{FP_linear}.
\end{theorem}

\begin{proof}
Suppose that $m_1,m_2$ are two weak solutions of \eqref{FP_linear}. Denote $\bar m=m_1-m_2$. Let $f\in C_b([0,T]\times H)$ satisfy \eqref{eq:KLipwf} and \eqref{eq:KLipa3}. In the first step, we are going to show that, after approximating $f$ and $w$, we can solve equation \eqref{equation_1} with solutions that are belong to the class of test functions $\mathcal{D}_T$ for equation \eqref{FP_linear}. After testing with these solutions, we pass to the limit in the approximation. In the second step, we show that the class of functions $f\in C_b([0,T]\times H)$ that satisfy \eqref{eq:KLipwf} and \eqref{eq:KLipa3} is measure-determining.

\textbf{Step 1:} In order to apply Lemma \ref{lem:Kolmogorov.approx}, we are going to approximate $w$ by functions that satisfy \eqref{eq:KLipa3}. We define $w_N(t,x) := w(t,P_N x)$, where $P_N$ is the orthogonal projection onto ${\rm span}\{e_1,\dots,e_N\}$. We have $w_N\in C_b([0,T]\times H,H)$, $\| w_N \|_\infty\leq\|w\|_\infty$ for all $N\in\mathbb{N}$, each $w_N$ satisfies Assumption \ref{Assumption:1} with the same constant $C_1$, and $w_N \to w$ uniformly on bounded subsets of $[0,T]\times H$ as $N\to \infty$. Using the same identification between $H_N$ and $\mathbb{R}^N$ as in the proof of Lemma \ref{lem:Kolmogorov.approx}, we can write
\[
w_N(t,x)=w_N(t,x^1,\dots,x^N).
\]
%for some function $\tilde{w}_N : [0,T]\times \mathbb{R}^N \to H$. 
Next, let $\rho \in C^{\infty}({\mathbb{R}}^{N})$ be a standard mollifier with $\text{supp}(\rho) \subset B_1(0)$, $\rho(x) \geq 0$, $\int_{{\mathbb{R}}^{N}} \rho(x) \mathrm{d}x = 1$, and let for $\eta>0$, $\rho_\eta(x) = (1/\eta^{N}) \rho(x/\eta)$. Then, we define $\tilde{w}_N^{\eta}:[0,T]\times \mathbb{R}^N \to H$ via
\begin{equation}
    \tilde{w}_N^{\eta}(t,x) = \int_{\mathbb{R}^{N+1}} w_N(t,y) \rho_{\eta}(t-s,x-y) \mathrm{d}s \mathrm{d}y \quad \text{for } t\in [0,T], x\in \mathbb{R}^N
\end{equation}
and then $w^{\eta}_N : [0,T]\times H \to H$ via
\begin{equation}
    w_N^{\eta}(t,x) = \tilde{w}_N^{\eta}(t,P_N x) \quad \text{for } t\in [0,T], x\in H.
\end{equation}

\iffalse
Then, we define $w_N^{\eta}:[0,T]\times H \to H$ via
\begin{equation}
\begin{split}
    w_N^{\eta}(t,x) &= \int_{\mathbb{R}^N} \tilde{w}_N(t,y^1,\dots,y^N) \prod_{i=1}^N \rho_{\eta}(x^i-y^i) \mathrm{d}y^1 \cdots \mathrm{d}y^N\\
    &=  \int_{\mathbb{R}^N} \tilde{w}_N(t,x^1-z^1,\dots,x^N-z^N) \prod_{i=1}^N \rho_{\eta}(z^i) \mathrm{d}z^1 \cdots \mathrm{d}z^N
\end{split}
\end{equation}
where $\rho_{\eta}(x) = (1/\eta) \rho(x/\eta)$ for $\eta>0$ are standard mollifiers.
\fi
    
By standard properties of finite dimensional mollification, we have $\| w_N^\eta \|_\infty\leq\|w_N\|_\infty\leq \|w\|_\infty$ for all $N\in\mathbb{N},\eta>0$, each $w_N^\eta$ satisfies Assumption \ref{Assumption:1} with the same constant $C_1$, and $w_N^\eta \to w_N$ uniformly on $[0,T]\times H$ as $\eta\to 0$. Moreover, a direct computation shows
\begin{equation}\label{Dw_m_delta_Lipschitz}
        \| ( Dw^{\eta}_N(t,x) - Dw^{\eta}_N(t,y) )(x-y) \|_{-1} \leq C_{N,\eta} \|x-y\|_{-1}^2
    \end{equation}
    for some constant $C_{N,\eta}$ depending on $N$ and $\eta$, i.e., the $w^{\eta}_N$ satisfy \eqref{eq:KLipa3}.
    
    \iffalse
    To show this, we note that $Dw_N^{\eta} : [0,T]\times H \to L(H)$ is given by
    \begin{equation}
    \begin{split}
        Dw^{\eta}_N(t,x) &= \sum_{k=1}^N \int_{\mathbb{R}^N} \tilde{w}_N(t,y^1,\dots,y^N) \rho'_{\eta}(x^k-y^k) \prod_{j=1,j\neq k}^m \rho_{\eta}(x^j-y^j) \mathrm{d}y^1 \cdots \mathrm{d}y^N \langle e_k,\cdot \rangle\\
        &= \sum_{k=1}^N \int_{\mathbb{R}^N} \tilde{w}_N(t,x^1-z^1,\dots,x^N-z^N) \rho'_{\eta}(z^k) \prod_{j=1,j\neq k}^N \rho_{\eta}(z^j) \mathrm{d}z^1 \cdots \mathrm{d}z^N \langle e_k,\cdot \rangle.
    \end{split}
    \end{equation}
    Thus,
    \begin{equation}
    \begin{split}
        &\|(Dw^{\eta}_N(t,x) - Dw^{\eta}_N(t,y))(x-y) \|_{-1}\\
        &= \Bigg \| \sum_{k=1}^N \int_{\mathbb{R}^N} ( \tilde{w}_N(t,x^1-z^1,\dots,x^N-z^N) - \tilde{w}_N(t,y^1-z^1,\dots,y^N-z^N) ) \rho'_{\eta}(z^k) \\
        &\qquad\qquad \times \prod_{j=1,j\neq k}^N \rho_{\eta}(z^j) \mathrm{d}z^1 \cdots \mathrm{d}z^N (x^k-y^k) \Bigg \|_{-1}\\
        &\leq C_{N,\eta} \|P_N( x-y) \|_{-1}\|P_N( x-y) \|\leq C_{N,\eta} \|x-y \|_{-1}^2
    \end{split}
    \end{equation}
    since $\|P_N( x-y) \|\leq C_N\|P_N( x-y) \|_{-1}\leq C_N\|x-y \|_{-1}$. 
      \fi

Now, for every $k\geq 1$ we take $N(k), \eta(k)$ such that $\|w(t,x)-w_{N(k)}^{\eta(k)}(t,x)\|\leq \frac{1}{k}$ for all $t\in[0,T], \|x\|\leq k$.
Let for every $k\geq 1$, $\hat{u}_n^k\in \mathcal{D}_T$ be the functions from Lemma \ref{lem:Kolmogorov.approx}
 such that 
 \[
    \partial_t \hat{u}_n^k(t,x) + L_0 \hat{u}_n^k(t,x) - \langle w_{N(k)}^{\eta(k)}(t,x), D\hat{u}_n^k(t,x) \rangle = f_n^k(t,x) \quad\text{in}\,\,(0,T)\times H
\]
for some functions $f_n^k$ as in that lemma. We notice that due to \eqref{Bound_for_DH-1uhat}, the functions $D\hat{u}_n^k$ in particular satisfy $\|D\hat{u}_n^k\|_\infty\leq \|B^{1/2}\|_{L(H)} M$. Hence we have
\[
    \partial_t \hat{u}_n^k(t,x) + L_0 \hat{u}_n^k(t,x) - \langle w, D\hat{u}_n^k(t,x) \rangle = f_n^k(t,x)+\langle w_{N(k)}^{\eta(k)}(t,x)-w(t,x), D\hat{u}_n^k(t,x) \rangle \quad\text{in}\,\,(0,T)\times H.
\]
Using the definition of a weak solution of \eqref{FP_linear} with $\hat{u}_n^k\in \mathcal{D}_T$, we have
\begin{equation}\label{equation_1k}
    \int_0^T \left(\int_H[f_n^k(t,x)+\langle w_{N(k)}^{\eta(k)}(t,x)-w(t,x), D\hat{u}_n^k(t,x) \rangle]\bar m(t,\mathrm{d}x)\right)\mathrm{d}t=\int_H\hat{u}_n^k(T,x) \bar m(T,\mathrm{d}x).
\end{equation}
We estimate
\[
\begin{split}
&
\left|\int_0^T\left(\int_H \langle w_{N(k)}^{\eta(k)}(t,x)-w(t,x), D\hat{u}_n^k(t,x) \rangle\bar m(t,\mathrm{d}x)\right)\mathrm{d}t\right|
\\
&
\leq \frac{2TM}{k}+2\|w\|_\infty\|B^{\frac{1}{2}}\|_{L(H)}M(m_1([0,T]\times (H\setminus B_k(0)))+m_2([0,T]\times (H\setminus B_k(0))))=:\omega(k)\to 0.
\end{split}
\]
Therefore, noticing that $\hat{u}_n^k(T,\cdot)\to 0$ as $n\to\infty$ uniformly on bounded subsets of $H$, letting $n\to\infty$ in \eqref{equation_1k}, we obtain
\[
\left|\int_0^T \left(\int_Hf(t,x)\bar m(t,\mathrm{d}x)\right)\mathrm{d}t\right|\leq \omega(k).
\]
Sending $k\to\infty$ above finally gives
\begin{equation}\label{eq:integral0}
    \int_0^T \left(\int_H f(t,x)\bar m(t,\mathrm{d}x)\right)\mathrm{d}t=0
\end{equation}
for all $f\in C_b([0,T]\times H)$ satisfying \eqref{eq:KLipwf} and \eqref{eq:KLipa3}.

\textbf{Step 2:} Now, let us show that every function $g\in C_b([0,T]\times H)$ can be approximated in $L^1([0,T]\times H;m_1+m_2)$ by functions in $C_b([0,T]\times H)$ satisfying \eqref{eq:KLipwf} and \eqref{eq:KLipa3}. Certainly, $g$ can be approximated by functions in $C_b([0,T]\times H)$ with bounded support by multiplying $g$ by continuous cut-off functions. So, let us suppose that $g\in C_b([0,T]\times H)$ and has bounded support. Let the support of $g$ be contained in $[0,T]\times B_R(0)$. We set for $N\geq 1$, $f_N(t,x)=g(t,P_N x), (t,x)\in [0,T]\times H$. Then $\|f_N\|_\infty\leq\|g\|_\infty$, $f_N\to g$ pointwise, and hence $f_N\to g$ in $L^1([0,T]\times H;m_1+m_2)$. Since $B$ is compact, $[0,T]\times B_R(0)$ is compact in $[0,T]\times H_{-1}$. Moreover, $\|P_Nx\|\leq C_N\|x\|_{-1}$ so $f_N$ is uniformly continuous in the $|\cdot|\times\|\cdot\|_{-1}$ norm. Thus $f_N$ can be extended to a function in $BUC([0,T]\times H_{-1})$. We then define for $k\geq 1$
\[
f_N^k(t,x):=\inf_{z\in H_{-1}}\sup_{y\in H_{-1}}\left\{f_N(t,y)-\frac{1}{2k}\|z-y\|_{-1}^2+\frac{1}{k}\|x-z\|_{-1}^2\right\}.
\]
The functions $f_N^k$ belong to $BUC([0,T]\times H_{-1})$, hence in particular to $C_b([0,T]\times H)$, they satisfy \eqref{eq:KLipwf}, \eqref{eq:KLipa3} and they converge uniformly to $f_N$ as $k\to\infty$, see \cite{lasry_lions_1986} or \cite[Appendix D.3]{fabbri_gozzi_swiech_2017}. Thus, \eqref{eq:integral0} holds for $f_N^k$. Therefore, it holds for $f_N$ and hence for $g$. We thus conclude that \eqref{eq:integral0} holds for all $g\in C_b([0,T]\times H)$ and finally, by density, for all characteristic functions of sets in ${\mathcal B}([0,T]\times H)$. Therefore $m_1=m_2$.
\end{proof}

\subsection{Existence}

The existence of a weak solution of equation \eqref{FP_linear} is shown via the corresponding SDE
\begin{equation}\label{SDE:1}
    \mathrm{d}X(t) = [ AX(t) - w(t,X(t)) ] \mathrm{d}t + \sigma(X(t)) \mathrm{d}W(t),\quad \mathcal{L}(X_0) = m_0.
\end{equation}

  \begin{theorem}\label{th:mildSDE}
Let $m_{0}\in\mathcal{P}_{2}(H)$ and let Assumptions \ref{ass:HJB}(A1),(A4) and Assumption \ref{Assumption:1} hold. Then, there exists a unique mild solution of  \eqref{SDE:1} in the following sense:
\begin{itemize}
\item[(i)]There is a filtered probability space $(\Omega, \mathcal{F}, (\mathcal{F}_{t})_{t\in[0,T]}, {\mathbb P})$, an $\mathcal{F}_{0}$-measurable $H$-valued random variable $X_{0}$ with $\mathcal{L}(X_{0})=m_{0}$, a $\Xi$-valued cylindrical Wiener process $(W(t))_{t\in[0,T]}$ with respect to $(\mathcal{F}_{t})_{t\in[0,T]}$, and an adapted $H$-valued process $(X^{w}(t))_{t\in[0,T]}$  with continuous trajectories such that
\[
X^{w}(t)=\mathrm{e}^{tA}X_0-\int_{0}^{t} \mathrm{e}^{(t-s)A} w(s,X^{w}(s)) \mathrm{d}s + \int_{0}^{t} \mathrm{e}^{(t-s)A} \sigma(X^{w}(s))\mathrm{d} W(s), \ \ \ \forall \,t\in [0,T];
\]
\item[(ii)] If another mild solution $\widehat{X}^{w}$ exists, then $\mathcal{L}(X^{w}(t))=\mathcal{L}(\widehat{X}^{w}(t))$ for every $t\in[0,T]$.
\end{itemize}
\end{theorem}

\begin{proof}
The construction of a filtered probability space and a random variable $X_0$ is shown in the proof of \cite[Theorem 2.4]{federico_gozzi_swiech_2026}. Given a filtered probability space the existence of a unique mild solution is standard, see e.g. \cite[Theorem 7.2]{daprato_zabczyk_2014}. Uniqueness in law is proved for instance in \cite[Proposition 1.137]{fabbri_gozzi_swiech_2017}.
\end{proof}

\begin{lemma}\label{lem:mildSDEcontinuity}
Let $m_{0}\in\mathcal{P}_{2}(H)$. Let Assumptions \ref{ass:HJB}(A1),(A4) and Assumption \ref{Assumption:1} hold and let $X^{w}$ be the mild solution of \eqref{SDE:1}. Then, the following holds:
 
 \begin{itemize}
\item[(i)]  There is a constant $C_0$ depending on $T$ such that
 \begin{equation}\label{eq:moment}
\mathbb{E}\left[\sup_{t\in[0,T]}\|X^{w}(t)\|^2\right]\leq C_0\left(\int_H \|x\|^2 m_0(\mathrm{d}x)+\|w\|_{\infty}^2+\|\sigma\|_{\infty}^2\right)=:c_{m_0,\|w\|_{\infty}}.
  \end{equation}
  In particular, for all $t\in[0,T]$,
  \begin{equation}\label{eq:moment1}
\int_H \|x\|^2 {\mathcal{L}}(X^{w}(t))(\mathrm{d}x)\leq c_{m_0,\|w\|_{\infty}}.
  \end{equation}
  \item[(ii)] Let $m(t):={\mathcal{L}}(X^{w}(t))$, $t\in [0,T]$. Then $m\in \hat S$.
\item[(iii)]
 There is a constant $M_1$ depending on $\mathbb{E} [ \|X_0\|^2 ]$, $T$, $\|w\|_{\infty}$, $\|\sigma\|_{\infty}$, such that
  \begin{equation}\label{eq:cont1}
 {\mathbf{d}_{1,-1}}({\mathcal{L}}(X^{w}(t)),{\mathcal{L}}(X^{w}(s)))\leq M_1\sqrt{|t-s|}\quad \forall \,0\leq s,t\leq T.
\end{equation}
\end{itemize}
\end{lemma}
  
\begin{proof}
       (i) We have
    \begin{equation}
        \| X^w(t) \|^2 \leq C \| \mathrm{e}^{tA} X_0 \|^2 + C \left \| \int_0^t \mathrm{e}^{(t-s)A} w(s,X^w(s)) \mathrm{d}s \right \|^2 + C \left \| \int_0^t \mathrm{e}^{(t-s)A} \sigma(X^w(s)) \mathrm{d}W(s) \right \|^2.
    \end{equation}
    Moreover, by \cite[Proposition 1.112]{fabbri_gozzi_swiech_2017}
    \begin{equation}
    \begin{split}
        \mathbb{E} \left [ \sup_{t\in [0,T]} \left \| \int_0^t \mathrm{e}^{(t-s)A} \sigma(X^w(s)) \mathrm{d}W(s) \right \|^2 \right ] \leq C \mathbb{E} \left [ \int_0^T \| \sigma(X^w(s)) \|_{L_2(\Xi,H)}^2 \mathrm{d}s \right ].
    \end{split}
    \end{equation}
    Thus,
    \begin{equation}
    \begin{split}
        \mathbb{E} \left [ \sup_{t\in [0,T]} \| X^w(t) \|^2 \right ] \leq C \mathbb{E} \left [ \| X_0 \|^2 \right ] + C \mathbb{E} \left [ \int_0^T \| w(s,X^w(s)) \|^2 \mathrm{d}s \right ] + C \mathbb{E} \left [ \int_0^T \| \sigma(X^w(s)) \|_{L_2(\Xi,H)}^2 \mathrm{d}s \right ].
    \end{split}
    \end{equation}
    Due to the boundedness of $w$ and $\sigma$, this concludes the proof of (i).
    
    (ii) We have
    \begin{equation}
        \mathbf{d}^2_1(m(t),m(s)) \leq \mathbb{E} \left [ \| X^w(t) - X^w(s) \|^2 \right ].
    \end{equation}
    Since $(X^w(t))_{t\in [0,T]}$ has a continuous modification, the claim follows from (i) and the dominated convergence theorem.

    (iii) We first note that
    \begin{equation}
        \mathbf{d}^2_{1,-1}(\mathcal{L}(X^w(t)),\mathcal{L}(X^w(s))) \leq \mathbb{E} \left [ \| X^w(t) - X^w(s) \|_{-1}^2 \right ].
    \end{equation}
    Assume without loss of generality that $s<t$. Then, by It\^o's formula \cite[Proposition 1.164]{fabbri_gozzi_swiech_2017}, we have
    \begin{equation}\label{Holder_continuity_SDE_equation_1}
    \begin{split}
        &\mathbb{E} \left [ \| X^w(t) - X^w(s) \|_{-1}^2 \right ]\\
        &= 2 \int_s^t \mathbb{E} \left [ \langle X^w(r), A^* B(X^w(r)-X^w(s)) \rangle - \langle w(r,X^w(r)), B(X^w(r) - X^w(s)) \rangle \right ] \mathrm{d}r\\
        &\quad + \int_s^t \mathbb{E} \left [ \operatorname{Tr} \left ( \sigma(X^w(r)) \sigma(X^w(r))^* B \right ) \right ] \mathrm{d}r.
    \end{split}
    \end{equation}
    We remark that \cite[Proposition 1.164]{fabbri_gozzi_swiech_2017} is stated for a deterministic initial condition, however it is clear from its proof that it is also true for a random initial condition. For the first term on the right-hand side, we obtain
    \begin{equation}
    \begin{split}
        &\left | \int_s^t \mathbb{E} \left [ \langle X^w(r), A^* B(X^w(r)-X^w(s)) \rangle - \langle w(r,X^w(r)), B(X^w(r) - X^w(s)) \rangle \right ] \mathrm{d}r \right |\\
        &\leq \int_s^t \mathbb{E} \left [ \| X^w(r) \| \| X^w(r) - X^w(s) \| + \| w(r,X^w(r)) \| \| X^w(r) - X^w(s) \| \right ] \mathrm{d}r\\
        &\leq C \mathbb{E} \left [ \sup_{r\in [0,T]} \| X^w(r) \|^2 \right ] (t-s).
    \end{split}
    \end{equation}
    For the second term on the right-hand side of \eqref{Holder_continuity_SDE_equation_1}, since $\sigma$ is bounded, we obtain
    \begin{equation}
        \left | \int_s^t \mathbb{E} \left [ \operatorname{Tr} \left ( \sigma(X^w(r)) \sigma(X^w(r))^* B \right ) \right ] \mathrm{d}r \right | \leq C (t-s).
    \end{equation}
    Now, the claim follows from part (i).
\end{proof}
 
We remark that in \cite{federico_gozzi_swiech_2026} a counterpart of the uniform continuity estimate (iii) with respect to the $\mathbf{d}_{1}$ distance was proved using compactness of the semigroup generated by $A$.
 
\begin{theorem}\label{th:uniqFP}
    Let Assumptions \ref{ass:HJB}(A1) and \ref{Assumption:1} hold, let $m_{0}\in\mathcal{P}_{2}(H)$ and let $X^{w}$ be the mild solution of \eqref{SDE:1}. Then, $m^w(t):={\mathcal{L}}(X^{w}(t))$ is the unique weak solution of \eqref{FP_linear} in $S_{FP}$. Moreover, $m^w\in S\cap\hat S$.
\end{theorem}

\begin{proof}
    The fact that $m^w(t):={\mathcal{L}}(X^{w}(t))$ is a weak solution of \eqref{FP_linear} is a direct consequence of It\^{o}'s formula, see
    \cite[Proposition 1.164]{fabbri_gozzi_swiech_2017}.
    Uniqueness follows from Theorem \ref{th:FPuniqueness}. By Lemma \ref{lem:mildSDEcontinuity}, $m^w\in S\cap \hat S$.
\end{proof}
  
%  \begin{lemma}\label{lem:mildSDEcompactness}
 % Let Assumption \ref{Assumption:1} hold and let $m^{w}$ be the weak solution of
 %\eqref{FP_linear}. Then
% \begin{equation}\label{eq:tails}
% \sup_{t\in[0,T]}\int_{H_{-1}}\sum_{k=N+1}^\infty\langle x,{\tilde e}_k\rangle_{-1}^2 m^w(t)(\mathrm{d}x)\leq \|Q_N B^{\frac 12}\|c_{m_0,\|w\|_{\infty}}
 %\end{equation}
 %for every $N\geq 1$.
  %  \end{lemma}

%\section{The HJB Equation: Viscosity Solutions}

%For given $m : [0,T]\to\mathcal{P}_1(H)$, let $v : [0,T]\times H\to \mathbb{R}$ be the viscosity solution of \eqref{HJB}. Then, under certain %assumptions, we know that $v(t,\cdot) \in C^{1,1}(H_{-1})$, for all $t\in [0,T]$, see \cite{defeo_swiech_wessels_2025}.

\section{The MFG System: Existence of Solutions}\label{sec:existence}

In this section we assume that Assumption \ref{ass:HJB} is satisfied and $m_0\in \mathcal{P}_2(H)$.

\begin{definition}\label{definition_MFG_solution}
    A pair $(v,m)$ is a solution of the MFG system \eqref{HJB}--\eqref{FP} if $v$ is a viscosity solution of \eqref{HJB} satisfying $v(t,\cdot)\in C^{1,1}_b(H_{-1})$ for all $t\in [0,T]$, and $m$ is a weak solution of \eqref{FP_linear} in the sense of Definition \ref{definition_weak_solution}, where $w(t,x)= \mathcal{H}_p(x,Dv(t,x),m(t))$ satisfies Assumption \ref{Assumption:1}.
\end{definition}

In order to prove the existence of a solution of the MFG system, we carefully adapt the classical strategy based on Tikhonov's fixed point theorem to our setting. Let us first introduce the relevant spaces.

Let $\hat M_1$ be the constant from Assumption \ref{ass:HJB}(A6). We set
\[
\|\mathcal{H}_p\|_{\infty}:=\sup \left \{\|\mathcal{H}_p(x,p,\mu)\|:x,p\in H, \mu\in\mathcal{P}_1(H), \|p\|\leq \hat M_1 \right \}.
\]
We define the set $Q_{m_0}$ by setting
\[
Q_{m_0}:=\left \{\mu\in \mathcal{P}_1(H): 
\int_H \|x\|^2 \mu(\mathrm{d}x)\leq c_{m_0,\|\mathcal{H}_p\|_{\infty}}\right \},
\]
where $c_{m_0,\|\mathcal{H}_p\|_{\infty}}$ is the constant from \ref{lem:mildSDEcontinuity}(i) with $\|w\|_{\infty}$ replaced by $\| \mathcal{H}_p \|_{\infty}$.

\begin{lemma}
    Let Assumption \ref{ass:HJB} be satisfied and $m_0\in \mathcal{P}_2(H)$. Then, the set $Q_{m_0}$ is compact in $\mathcal{P}_1(H_{-1})$.
\end{lemma}

\begin{proof}
Since for every $R>0$, the closed ball $\overline B_R(0)\subset H$ is compact in $H_{-1}$, $Q_{m_0}$ is tight in $\mathcal{P}(H_{-1})$. Moreover, $Q_{m_0}$ has uniformly integrable first moments (see \cite[Equation (5.1.20)]{ambrosio_gigli_savare_2008}), and hence, by \cite[Proposition 7.1.5]{ambrosio_gigli_savare_2008}, it is relatively compact in $\mathcal{P}_1(H_{-1})$.

Let us show that $Q_{m_0}$ is closed. To this end, let $(\mu_n)_{n\in\mathbb{N}} \subset Q_{m_0}$ be a sequence in $\mathcal{P}_1(H_{-1})$ converging to $\mu\in \mathcal{P}_1(H_{-1})$. We need to show that $\mu \in \mathcal{P}_1(H)$, and
\begin{equation}\label{second_moment_condition}
    \int_H \| x\|^2 \mu(\mathrm{d}x) \leq c_{m_0,\| \mathcal{H}_p\|_{\infty}}.
\end{equation}
For a compact set $K\subset H_{-1}\setminus H$ and $R>0$, we can find a continuous function $f:H_{-1} \to \mathbb{R}$ such that $0\leq f \leq 1$, $f\equiv 1$ on $K$, and $f\equiv 0$ on $B_R(0)$. First, we note that
\begin{equation}
    \int_{H_{-1}} f(x) \mu_n(\mathrm{d}x) \leq \mu_n(H_{-1} \setminus B_R(0)) \leq \frac{1}{R^2} \int_H \| x\|^2 \mu_n(\mathrm{d}x) \leq \frac{c_{m_0,\|\mathcal{H}_p\|_{\infty}}}{R^2}.
\end{equation}
Thus,
\begin{equation}
    \mu(K) \leq \int_{H_{-1}} f(x) \mu(\mathrm{d}x) = \lim_{n\to\infty} \int_{H_{-1}} f(x) \mu_n(\mathrm{d}x) \leq \frac{c_{m_0,\|\mathcal{H}_p\|_{\infty}}}{R^2}.
\end{equation}
Taking the limit $R\to \infty$ shows that $\mu(K)=0$ for all compact sets $K \subset H_{-1}\setminus H$. Hence, $\mu \in \mathcal{P}_1(H)$. Moreover, we have
\begin{equation}
\begin{split}
    \int_{H} \| x\|^2 \mu(\mathrm{d}x) &= \sup_{N\in \mathbb{N}} \int_H ( \| P_N x\|^2 \wedge N ) \mu(\mathrm{d}x) = \sup_{N\in\mathbb{N}} \lim_{n\to\infty} \int_H (\|P_N x\|^2 \wedge N) \mu_n(\mathrm{d}x)\\
    &\leq \lim_{n\to\infty} \sup_{N\in\mathbb{N}} \int_H ( \|P_N x\|^2 \wedge N) \mu_n(\mathrm{d}x) = \lim_{n\to\infty} \int_H \|x\|^2 \mu_n(\mathrm{d}x) \leq c_{m_0,\| \mathcal{H}_p \|_{\infty}},
\end{split}
\end{equation}
which shows \eqref{second_moment_condition}, and hence $Q_{m_0}$ is compact.
\end{proof}

Next, we define 
\begin{equation}
C_{m_0}:= \left \{m:[0,T]\to Q_{m_0}: m(0)=m_0\,\,\text{and}\,\,  \mathbf{d}_{1,-1}(m(t),m(s))\leq M_1\sqrt{|t-s|}\right \},
\end{equation}
where $M_1$ is the constant from Lemma \ref{lem:mildSDEcontinuity}(iii) with $\|w\|_\infty$ replaced by $\|\mathcal{H}_p\|_{\infty}$. Note that $C_{m_0} \subset S$.

\begin{lemma}\label{lemma:CC}
Let Assumption \ref{ass:HJB} be satisfied and $m_0\in \mathcal{P}_2(H)$. Then, the set $C_{m_0}$ is convex and compact in $(\tilde{S},\rho_{-1,\infty})$.
\end{lemma}

\begin{proof}
    Let $m_1,m_2\in C_{m_0}$, $\lambda\in [0,1]$. Noting that the $\mathbf{d}_{1,-1}$-Wasserstein distance between two measures coincides with the optimal transport cost associated with the cost function $c:H_{-1}\times H_{-1} \to \mathbb{R}$, $c(x,y) = \| x-y\|_{-1}$, we deduce from \cite[Theorem 4.8, page 47]{villani_2009} or the discussion at the beginning of \cite[Section 7.4]{villani_2003}
    \begin{equation}
    \begin{split}
        &\mathbf{d}_{1,-1}(\lambda m_1(t) + (1-\lambda) m_2(t), \lambda m_1(s) + (1-\lambda) m_2(s))\\
        &\leq \lambda \mathbf{d}_{1,-1}(m_1(t),m_1(s)) + (1-\lambda) \mathbf{d}_{1,-1}(m_2(t),m_2(s)) \leq M_1 \sqrt{|t-s|},
    \end{split}
    \end{equation}
    which implies the convexity of $C_{m_0}$.

     The relative compactness of $C_{m_0}$ follows from compactness of $Q_{m_0}$ and the Arzel\`a--Ascoli's theorem for functions with values in complete metric spaces.
    
 Finally, let us show that $C_{m_0}$ is closed. To this end, let $\{m_{n}\}_{n\in\mathbb{N}}\subset C_{m_0}$ be such that $m_{n}\to m\in \tilde{S}$. This means that  
 \[
 \sup_{t\in [0,T]}\mathbf{d}_{1,-1}(m_{n}(t), m(t))\to 0.
 \]
 Obviously, $m(0)=m_0$ and since $Q_{m_0}$ is closed, we have $m(t)\in Q_{m_0}$ for every $t\in [0,T]$. The inequality $\mathbf{d}_{1,-1}(m(t),m(s))\leq M_1\sqrt{|t-s|}$ is also obvious by uniform convergence. Therefore $C_{m_0}$ is closed in 
 $(\tilde{S},\rho_{-1,\infty})$.
 \end{proof}
To apply Tikhonov's fixed point theorem we follow the strategy from  \cite{federico_gozzi_swiech_2026}.
We embed $C_{m_0}$ into the locally convex topological vector space $C([0,T];\mathcal{M}(H_{-1}))$, endowed with the topology induced by the family of seminorms 
\[
\left|m(\cdot)\right|_{f}:=\sup_{t\in[0,T]}\left|\int_{H_{-1}} f(x) m(t,\mathrm{d}x)\right|, \ \ \ \ f\in C_{b}(H_{-1}).
\]
We call this topology on $C([0,T];\mathcal{M}(H_{-1}))$ the {\it uniform-weak topology} and denote by $\tau_{\mathbf{uw}}$. The following lemma is proved exactly as \cite[Lemma 5.4]{federico_gozzi_swiech_2026}.

\begin{lemma}\label{lem:topology}
Let Assumption \ref{ass:HJB} be satisfied and $m_0\in \mathcal{P}_2(H)$. Then, the following inclusion holds:
\[
C_{m_0} \ \subset \ C([0,T];\mathcal{M}(H_{-1})).
\]
Moreover, the topology induced by $\rho_{-1,\infty}$ on $C_{m_0}$ is the same as the one induced   by $\tau_{\mathbf{uw}}$.
\end{lemma}

Now, we are in a position to state and prove the main result of this section.

\begin{theorem}\label{teo:existence}
Let Assumption \ref{ass:HJB} be satisfied and $m_0\in \mathcal{P}_2(H)$.  Then, the MFG system \eqref{HJB}--\eqref{FP} has a solution. 
\end{theorem}

\begin{proof}
We apply Tikhonov's fixed theorem to the following map $\Psi$. Given $m\in C_{m_0}$, we first solve the HJB equation \eqref{HJB} to obtain a corresponding viscosity solution $v^{(m)}$ which, by Assumption \ref{assumption2}(A6), extends to a function $v^{(m)}$ such that $v^{(m)}(t,\cdot) \in C^{1,1}_b(H_{-1})$. Now, given $m(t)$ and $v^{(m)}$, we solve equation \eqref{FP_linear} with $w(t,x) = \mathcal{H}_p(x,Dv^{(m)}(t,x),m(t))$. We denote its solution by $\Psi(m)$. Theorem \ref{th:uniqFP} and Lemma \ref{lem:mildSDEcontinuity} guarantee that the map $\Psi:C_{m_0} \to C_{m_0}$, $m\mapsto \Psi(m)$ is well-defined.

We need to show that $\Psi:C_{m_0} \to C_{m_0}$ has a fixed point.
We consider $C_{m_0}$ embedded in the topological vector space $(C([0,T];\mathcal{M}(H_{-1})),\tau_{\mathbf{uw}})$. By 
Lemma \ref{lem:topology} we can indifferently use on $C_{m_0}$ the $\tau_{\mathbf{uw}}$ or the $\rho_{-1,\infty}$ topology. Since $C_{m_0}$ is compact and convex, to prove the existence of a fixed point of $\Psi$ in 
$C_{m_0}$ we need to show that $\Psi|_{C_{m_0}}$ is continuous.

Let $(m_{n})\subset C_{m_0}$ be such that $m_{n}\to m\in C_{m_0}$.  Since $C_{m_0}$ is compact, from each subsequence $(m_{n_{k}})$, one can extract a sub-subsequence such that $\Psi(m_{n_{k_h}})\to \hat{m}$ for some $\hat{m}\in C_{m_0}$.
 Set
 \[
 w^{(m)}:=\mathcal{H}_{p}(\cdot,Dv^{(m)}(\cdot,\cdot), m(\cdot)), \ \ \ \ w^{(m_{n_{k_{h}}})}:=\mathcal{H}_{p}(\cdot,Dv^{(m_{n_{k_{h}}})}(\cdot,\cdot), m_{n_{k_{h}}}(\cdot)).
 \]
For each $h\in\mathbb{N}$,
\begin{align}
&
\int_{H} \varphi (t,x) \Psi(m_{n_{k_{h}}})(t,\mathrm{d}x) - \int_{H} \varphi (0,x) m_0(\mathrm{d}x)\\&
= \int_0^{t} \left(\int_{H}\left[\partial_{t}\varphi(s,x)+L_{0}\varphi(s,x)-\langle w^{(m_{n_{k_{h}}})}(s,x), D\varphi(s,x)\rangle \right]\Psi(m_{n_{k_{h}}})(s,\mathrm{d}x)\right) \mathrm{d}s
\end{align}
for every $t\in(0,T]$ and $\varphi\in \mathcal{D}_{T}$. Since $m_{n_{k_h}}\to m$ in $C_{m_0}$, we have by Remark \ref{remark_convergence_Dv}(iii), that $Dv^{(m_{n_{k_{h}}})}\to Dv^{(m)}$ uniformly on $[0,T]\times H$ and thus $w^{(m_{n_{k_{h}}})}\to w^{(m)}$ uniformly on $[0,T]\times H$ as $h\to\infty$. Thus if we let $h\to\infty$ in the equation above, we obtain
\begin{align}
&
\int_{H} \varphi (t,x) \hat m(t,\mathrm{d}x) - \int_{H} \varphi (0,x) m_{0}(\mathrm{d}x)\\&
= \int_0^{t} \left(\int_{H}\left[\partial_{t}\varphi(s,x)+L_{0}\varphi(s,x)-\langle w^{(m)}(s,x), D\varphi(s,x)\rangle \right]\hat m(s,\mathrm{d}x)\right)\mathrm{d}s.
\end{align}
On the other hand, we also have
\begin{align}
&
\int_{H} \varphi (t,x) \mathcal{L}(X^{(m)}(\cdot))(t,\mathrm{d}x) - \int_{H} \varphi (0,x) m_{0}(\mathrm{d}x)\\&
= \int_0^{t} \left(\int_{H}\left[\partial_{t}\varphi(s,x)+L_{0}\varphi(s,x)-\langle w^{(m)}(s,x), D\varphi(s,x)\rangle \right]\mathcal{L}(X^{(m)}(\cdot))(s,\mathrm{d}x)\right)\mathrm{d}s.
\end{align}
Hence, by Theorem \ref{th:uniqFP}, we must have $\hat m=\mathcal{L}(X^{(m)}(\cdot))$.
By the arbitrariness of the choice of the subsequence $n_{k}$, it then follows that  
\[
\Psi(m_{n})\to \hat m =\mathcal{L}(X^{(m)}(\cdot))=\Psi(m) \ \ \ \text{in} \ (\tilde{S},\rho_{-1,\infty}),
\]
which implies the continuity of $\Psi$.
\end{proof}

\section{The MFG System: Uniqueness of Solutions}\label{section_MFG_Uniqueness}
To prove uniqueness of solutions to the MFG system \eqref{HJB}--\eqref{FP} we follow the approach used in \cite{federico_gozzi_swiech_2026}, which in turn was an adaptation to the Hilbert space case of a standard proof employed in the finite dimensional case, see e.g. \cite[Theorem 1.4]{cardaliaguet_porretta_2020}. However, instead of approximating mild solutions of Kolmogorov equations by smooth solutions of approximate equations as it was done in \cite{federico_gozzi_swiech_2026}, we directly approximate viscosity solutions of the HJB equations. We use a variant of the Lasry--Lions monotonicity condition from \cite[Theorem 1.4]{cardaliaguet_porretta_2020} (see also \cite{lasry_lions_2006,lasry_lions_2006_2,lasry_lions_2007}).
Regarding other monotonicity conditions that guarantee uniqueness of solutions for finite-dimensional MFG systems we refer to \cite{graber_meszaros_2023,meszaros_mou_2024} and the references therein. We also assume that the Hamiltonian $\mathcal{H}$ has a separated structure as described below. Well-posedness of MFG systems which do not have separated form has been studied in \cite{meszaros_mou_2024}. 
\begin{assumption}
\label{hp:uniqueness}
\begin{itemize}\smallskip
\item[]
  \item[(i)]
The Hamiltonian $\mathcal{H}$ has a separated form, i.e.
\[
\mathcal{H}(x,p,\mu)=\mathcal{H}^0(x,p)-F(x,\mu),
\]
for some continuous  functions
\[
\mathcal{H}^0:H \times H \to \R,
\qquad F: H\times \mathcal{P}_1(H)\to \R.
\]
  \item[(ii)]
The map $\mathcal{H}^0$ is convex in $p$.
\medskip
  \item[(iii)]
The functions $F$ and $G$ are monotone in the $\mu$ variable in the following sense:
\[
\int_H [F(x,\mu_1)-F(x,\mu_2)](\mu_1-\mu_2)(\mathrm{d}x)
\ge 0
\quad
\forall \mu_1,\mu_2\in \mathcal{P}_1(H),
\]
\[
\int_H [G(x,\mu_1)-G(x,\mu_2)](\mu_1-\mu_2)(\mathrm{d}x) \ge 0
\quad
\forall \mu_1,\mu_2\in \mathcal{P}_1(H).
\]
\smallskip
\item[(iv)]
One of the following conditions is satisfied: 
\begin{itemize}
  \item[(a)] For all $\mu_1,\mu_2\in \mathcal{P}_1(H)$ such that $\mu_1\ne \mu_2$ it holds
\[
\int_H [F(x,\mu_1)-F(x,\mu_2)](\mu_1-\mu_2)(\mathrm{d}x)
> 0.
\]
  \item[(b)] For all $\mu_1,\mu_2\in \mathcal{P}_1(H)$ it holds
\begin{equation}
    \begin{cases}
        \int_H [F(x,\mu_1)-F(x,\mu_2)](\mu_1-\mu_2)(\mathrm{d}x)
= 0
\quad \Longrightarrow\quad F(\cdot,\mu_1)=F(\cdot,\mu_2)\\
\int_H [G(x,\mu_1)-G(x,\mu_2)](\mu_1-\mu_2)(\mathrm{d} x) = 0
\quad \Longrightarrow\quad G(\cdot,\mu_1)=G(\cdot,\mu_2).
    \end{cases}
\end{equation}
  \item[(c)] For all $x\in H$, $p_1,p_2\in H$ it holds
\[
\mathcal{H}^0(x,p_1)-\mathcal{H}^0(x,p_2)
-\langle \mathcal{H}^0_p(x,p_2), p_{1}-p_{2}\rangle =0 \quad \Longrightarrow\quad\mathcal{H}^0_p(x,p_1)=\mathcal{H}^0_p(x,p_2).
\]
\end{itemize}
\end{itemize}
\end{assumption}

\begin{theorem}
\label{th:uniqueness}
Let Assumptions \ref{ass:HJB} and \ref{hp:uniqueness} be satisfied and let $m_0\in \mathcal{P}_2(H)$.
Then, there exists a unique solution of
the MFG system \eqref{HJB}--\eqref{FP}.
\end{theorem}
\begin{proof}
We only need to show uniqueness of solutions as existence of a solution of \eqref{HJB}--\eqref{FP} was proved in Theorem \ref{teo:existence}.
Let  $(v_1,m_1)$ and $(v_2,m_2)$ be two solutions
of the MFG system \eqref{HJB}--\eqref{FP}. We set
$\bar{v}=v_1-v_2$ and $\bar{m}=m_1-m_2$.

Let $\hat{v}_n^1,\hat{v}_n^2\in \mathcal{D}_T$ be the functions from Lemma \ref{lem:HJB.approx} such that $\hat{v}_n^i\to v_i, D\hat{v}_n^i \to Dv_i, i=1,2$, uniformly on bounded subsets of $[0,T]\times H$ and such that 
\begin{equation}\label{equation_1HJBi}
    \partial_t \hat{v}_n^i(t,x) + L_0 \hat{v}_n^i(t,x) - \mathcal{H}(x,D\hat{v}_n^i(t,x),m_i(t)) = f_n^i(t,x) \quad\text{in}\,\,(0,T)\times H,
\end{equation}
where $f_n^i(t,x)\to 0$ for all $(t,x)\in (0,T)\times H$ and $|f_n^i(t,x)|\leq C(1+\|x\|)$ for some $C\geq 0$ independent of $n$. We denote $\bar v_n=\hat{v}_n^1-\hat{v}_n^2$ and
\[
g_n(t,x):=\mathcal{H}(x,D\hat{v}_n^1(t,x),m_1(t))-\mathcal{H}(x,D\hat{v}_n^2(t,x),m_2(t))+f_n^1(t,x)-f_n^2(t,x).
\]
The functions $\bar v_n$ satisfy on $(0,T)\times H$
\begin{equation}\label{eq:barvn1}
{\partial_t}{\bar v_n}(t,x)+ {L_0} {\bar v_n}(t,x)=g_n(t,x)
\end{equation}
while, since $\bar v_n\in \mathcal{D}_T$, we have
\begin{align}\label{aaabarn}
&
\int_{H} \bar{v}_n (T,x) \bar{m}(T,\mathrm{d}x)
=
\int_0^{T} \left(\int_{H}\left[\partial_{t}\bar{v}_n(t,x)+L_{0}\bar{v}_n(t,x) \right]\bar{m}(t,\mathrm{d}x)\right)\mathrm{d}t
\\&\nonumber
- \int_0^{T} \left(\int_{H} \langle\mathcal{H}^0_p(x,Dv_1(s,x)),\, D\bar{v}_n(t,x)\rangle m_1(t,\mathrm{d}x)
-\int_{H} \langle\mathcal{H}^0_p(x,Dv_2(t,x)),\, D\bar{v}_n(t,x)\rangle m_2(t,\mathrm{d}x)\right)\mathrm{d}t.
\end{align}
Integrating \eqref{eq:barvn1} over $[0,T]\times H$ with respect to the measure $\bar{m}(t)\mathrm{d}t$ and combining it with \eqref{aaabarn}, we obtain
\begin{align}\label{eq:aaabis}
&
\int_{H} \bar{v}_n (T,x) \bar{m}(T,\mathrm{d}x) -\int_{0}^{T}\left(\int_{H}g_n(t,x)
\bar{m}(t,\mathrm{d}x)\right)\mathrm{d}t
=
\\&\nonumber
- \int_0^{T} \left(\int_{H} \langle\mathcal{H}^0_p(x,Dv_1(t,x)), D\bar{v}_n(t,x)\rangle m_1(t,\mathrm{d}x)
- \int_{H} \langle\mathcal{H}^0_p(x,Dv_2(t,x)), D\bar{v}_n(t,x)\rangle m_2(t,\mathrm{d}x)\right)\mathrm{d}t.
\end{align}
Therefore, passing to the limit in \eqref{eq:aaabis} and using the Dominated Convergence Theorem, we arrive at
\begin{align}\label{eq:aaabis2}
&0=-
\int_{H} \bar{v} (T,x) \bar{m}(T,\mathrm{d}x) +\int_{0}^{T}\left(\int_{H}g(t,x)
\bar{m}(t,\mathrm{d}x)\right)\mathrm{d}t
\\&\nonumber
- \int_0^{T} \left(\int_{H} \langle\mathcal{H}^0_p(x,Dv_1(t,x)), D\bar{v}(t,x)\rangle m_1(t,\mathrm{d}x)
- \int_{H} \langle\mathcal{H}^0_p(x,Dv_2(t,x)), D\bar{v}(t,x)\rangle  m_2(t,\mathrm{d}x)\right)\mathrm{d}t,
\end{align}
where
\[
g(t,x):=\mathcal{H}(x,Dv_1(t,x),m_1(t))-\mathcal{H}(x,Dv_2(t,x),m_2(t)).
\]
Using the separated structure of $\mathcal{H}$ from Assumption \ref{hp:uniqueness}(i), we thus have
\begin{equation}
\begin{split}
0&=\int_{H} \left [ G(x,m_1(T)) - G(x,m_2(T)) \right ]  \bar{m}(T,\mathrm{d}x)+
\int_0^{T} \int_{H}[F(x,m_1(t))-F(x,m_2(t))]
\bar{m}(t,\mathrm{d}x)\mathrm{d}t\\
&\quad + \int_0^{T} \int_{H}\left[ \mathcal{H}^{0}(x,Dv_2(t,x))-\mathcal{H}^{0}(x,Dv_1(t,x))
-\langle\mathcal{H}^0_p(x,Dv_1(t,x)),- D\bar v(t,x)\rangle \right]m_1(t,\mathrm{d}x)\mathrm{d}t\\
&\quad + \int_0^{T} \int_{H}\left[ \mathcal{H}^{0}(x,Dv_1(t,x))-\mathcal{H}^{0}(x,Dv_2(t,x))
-\langle\mathcal{H}^0_p(x,Dv_2(t,x)), \, D\bar{v}(t,x)\rangle \right]m_2(t,\mathrm{d}x)\mathrm{d}t.
\end{split}
\end{equation}
We finish the proof by following the argument of \cite[Theorem 6.2]{federico_gozzi_swiech_2026}, the only difference being that we work with viscosity solutions instead of mild solutions. For completeness, we include this argument.

By Assumption \ref{hp:uniqueness}(ii)--(iii) all terms above are nonnegative, so they all must be equal to $0$.
We now conclude as follows:
\begin{itemize}
\item If case (a) of Assumption \ref{hp:uniqueness}(iv) holds, we have $m_{1}=m_{2}$; then uniqueness of viscosity solutions of equation \eqref{HJB} gives $v_1=v_2$.
\item If case (b) of Assumption \ref{hp:uniqueness}(iv) holds, we first use uniqueness of viscosity solutions of equation \eqref{HJB}, which holds since 
\[
F(x,m_1(t))=F(x,m_2(t)),  \quad G(x,m_1(T))=G(x,m_2(T)),
\]
to obtain $v_1=v_2$; then, we conclude that $m_{1}=m_{2}$ by uniqueness of weak solutions of \eqref{FP_linear} with $w(t,x):=\mathcal{H}^0_p(x,Dv_1(t,x))=\mathcal{H}^0_p(x,Dv_2(t,x))$.
\item If case (c) of Assumption \ref{hp:uniqueness}(iv) holds, we first obtain 
\[
\mathcal{H}^0_p(x,Dv_1(t,x))=\mathcal{H}^0_p(x,Dv_2(t,x)),
\]
$m_1(t)$- and $m_2(t)$-a.e. for $t\in[0,T]$. This implies $m_{1}=m_{2}$ by uniqueness of weak solutions of \eqref{FP_linear} with $w(t,x)=\mathcal{H}^0_p(x,Dv_1(t,x))$ (or equivalently with $w(t,x)=\mathcal{H}^0_p(x,Dv_2(t,x))$) and then conclude that $v_{1}=v_{2}$ using uniqueness of viscosity solutions of equation \eqref{HJB}. \qedhere
\end{itemize}
\end{proof}

\section{Examples}\label{sec:examples}

In this section, we will discuss some examples in which Assumption \ref{ass:HJB} is satisfied. In these examples, the MFG system \eqref{HJB}--\eqref{FP} arises in the search for Nash equilibria of mean field games. We will briefly sketch this connection. For a more detailed exposition of how the MFG system arises in this context, see \cite[Section 1.1.3]{cardaliaguet_delarue_lasry_lions_2019} or \cite[Section 3.1.5]{carmona_delarue_2018}.

For a given flow of measures $(m(s))_{s\in [t,T]}$, let us consider a standard stochastic control problem for which the value function is given by
\begin{equation}
    v(t,x) := \inf_{\alpha(\cdot) \in \mathcal{A}_t} \mathbb{E} \left [ \int_t^T f(X(s), \alpha(s), m(s)) \mathrm{d}s + G(X(T),m(T)) \right ]
\end{equation}
subject to
\begin{equation}\label{state_equation}
    \mathrm{d}X(s) = [ AX(s) + b(X(s),\alpha(s),m(s)) ] \mathrm{d}s + \sigma(X(s)) \mathrm{d}W(s),\quad X(t) = x \in H.
\end{equation}
Here $f:H\times \Lambda \times \mathcal{P}_1(H) \to \mathbb{R}$ and $G: H \times \mathcal{P}_1(H) \to \mathbb{R}$ are the running and terminal cost, respectively, $b:H\times\Lambda\times \mathcal{P}_1(H)\to H$ and $\sigma: H \to L_2(\Xi,H)$ are the drift and noise coefficient, respectively, $(W(s))_{s\in [0,T]}$ is a cylindrical Wiener process in $\Xi$, and $\mathcal{A}_t$ is the set of control processes with values in a subset $\Lambda_0$ of a real, separable Hilbert space $\Lambda$.

It is well known that, under appropriate assumptions, $v$ is the unique viscosity solution of the HJB equation
\begin{equation}
    \partial_t v(t,x) +Lv(t,x) - \mathcal{H}(x,Dv(t,x),m(t)) = 0,\quad v(T,\cdot) = G(\cdot,m(T))
\end{equation}
where the Hamiltonian is given by
\begin{equation}\label{ex:control}
    \mathcal{H}(x,p,\mu) = \sup_{\alpha\in\Lambda_0} \{ - \langle b(x,\alpha,\mu) , p \rangle - f(x,\alpha,\mu) \}.
\end{equation}
Moreover, it follows that the optimal control satisfies
\begin{equation}
    \alpha^*(t,x) = \argmax_{\alpha\in \Lambda_0} \{ - \langle b(x,\alpha,m(t)) , Dv(t,x) \rangle - f(x,\alpha,m(t)) \}
\end{equation}
and therefore, by differentiating \eqref{ex:control}, we conclude that the optimal drift in equation \eqref{state_equation} takes the form
\begin{equation}
    b(x,\alpha^*(s,x),m(s)) = - \mathcal{H}_p(x,Dv(s,x),m(s)).
\end{equation}
Writing down the associated Fokker--Planck equation leads to
\begin{equation}
    \partial_t m(t) - L^* m(t) - \operatorname{div} \left ( \mathcal{H}_p(x,Dv(t,x),m(t)) m(t) \right ) = 0, \quad m(0) = m_0.
\end{equation}
Thus, finding a solution of the MFG system \eqref{HJB}--\eqref{FP} corresponds to finding a Nash equilibrium of the mean field game.

\subsection{Examples for \texorpdfstring{$C^{1,1}$}{C^{1,1}} Regularity}\label{section_C11_regularity}

We recall a $C^{1,1}$ regularity result from \cite{defeo_swiech_wessels_2025} for viscosity solutions of HJB equations \eqref{HJB} where the Hamiltonian $\mathcal{H}$ is given by \eqref{ex:control}. Here $\Lambda_0$ is a bounded, convex subset of a real, separable Hilbert space $\Lambda$, $f:H\times \Lambda \times \mathcal{P}_1(H) \to \mathbb{R}, b:H\times\Lambda\times \mathcal{P}_1(H)\to H$. 
    
\begin{assumption}\label{bsigmalipschitzw}
	\begin{enumerate}[label=(\roman*)]
		\item %The function $b$ is bounded and 
		There exists a constant $C>0$ such that
		\[
			\| b(x,\alpha,\mu) - b(x^{\prime},\alpha^{\prime},\mu' ) \| \leq C (\|x-x^{\prime}\|_{-1} + \|\alpha-\alpha^{\prime}\|_{\Lambda} +\mathbf{d}_{1,-1}(\mu,\mu'))
		\]
		%\[
		%\|b(x,a,\mu)\|\leq C(1+\|a\|_\Lambda)
		%\]
		for all $x,x^{\prime}\in H, \mu,\mu'\in \mathcal{P}_1(H)$ and $\alpha,\alpha^{\prime}\in \Lambda$.
				% \item The function $\sigma$ is bounded and there exists a constant $C>0$ such that
 %   \[
  %      \|\sigma(x)-\sigma(x^{\prime}) \|_{L_2(\Xi,H)} \leq C \|x-x^{\prime} \|_{-1}
  %  \]
  %  for all $x,x^{\prime}\in H$.
     \item The function $f$ is continuous, it is bounded on $H\times\Lambda_0\times\mathcal{P}_1(H)$ and there exists a constant $C>0$ such that
    \[
      |f(x,\alpha,\mu) - f(x^{\prime},\alpha,\mu')| \leq C (\|x-x^{\prime}\|_{-1}+\mathbf{d}_{1,-1}(\mu,\mu'))
    \]
    for all $x,x^{\prime}\in H,\mu,\mu'\in \mathcal{P}_1(H)$ and $\alpha\in \Lambda$.
   %  \item There exists a constant $C>0$ such that
   % \[
     % |f(x,a,\mu)| \leq C (1+\|x\|_H+\|a\|_{\Lambda}^2)
   % \]
    %  for all $x\in H$ and $a\in \Lambda$ and
   % \[
 % \lim_{\|a\|\to\infty}  \frac{f(x,a,\mu)}{\|a\|_\Lambda}=+\infty\quad\mbox{uniformly in}\,\,x,\mu.
 % \]
Moreover, for every $(\alpha,\mu)\in \Lambda \times\mathcal{P}_1(H)$, the function $f(\cdot,\alpha,\mu)$ extends to a semiconcave function on $H_{-1}$ with the semiconcavity constant independent of $\alpha,\mu$, and there exist constants $C,\nu \geq 0$ such that the map
		\[
			H\times\Lambda_0 \ni (x,\alpha )\mapsto f(x,\alpha,\mu) + C\|x\|_{-1}^2 - \nu \|\alpha\|^2_{\Lambda}
		\]
		is convex.

   % \item There exists a constant $C>0$ such that
   % \[
   %     |g(x,\mu) - g(x^{\prime},\mu')| \leq C( \|x-x^{\prime}\|_{-1}+\mathbf{d}_{1,-1}(\mu,\mu'))
   % \]
   % for all $x,x^{\prime}\in H, \mu,\mu'\in \mathcal{P}_1(H)$. Moreover, for every $\mu\in \mathcal{P}_1(H)$, $g(\cdot,\mu)$ extends to a function %in $C^{1,1}(H_{-1})$ and $\|g(\cdot,\mu)\|_{C^{1,1}(H_{-1})}\leq C$ for some $C\geq 0$ independent of $\mu$.
  %  \item If $\Lambda_0$ is unbounded we assume that $f$ and $g$ are bounded from below.
\end{enumerate}
\end{assumption}

It is easy to see that Assumption \ref{bsigmalipschitzw} implies Assumption \ref{ass:HJB}(A2).

\begin{assumption}\label{bsigmafrechetw}
	\begin{enumerate}[label=(\roman*)]
		\item The function $b$ is Fr\'echet differentiable in $(x,\alpha)$ and there exists a constant $C>0$ such that
        \[
        \begin{split}
            &\| (D_x b(x,\alpha,\mu) - D_x b(x^{\prime},\alpha^{\prime},\mu))(x-x^{\prime}) + ( D_{\alpha} b(x,\alpha,\mu) - D_{\alpha} b(x^{\prime},\alpha^{\prime},\mu))(\alpha-\alpha^{\prime})\|_{-1}\\
            &\leq C \left ( \|x-x^{\prime}\|^2_{-1} + \| \alpha-\alpha^{\prime} \|_{\Lambda}^2 \right )
        \end{split}
        \]
        for all $x,x^{\prime}\in H,\mu\in \mathcal{P}_1(H)$, $\alpha,\alpha^{\prime}\in \Lambda$.
		\item The function $\sigma: H \to L_2(\Xi,H)$ is Fr\'echet differentiable and there exists a constant $C>0$ such that
		\[
			\| (D\sigma(x) - D\sigma(x^{\prime}))(x-x^{\prime})\|_{L_2(\Xi,H_{-1})} \leq C \|x-x^{\prime}\|_{-1}^2
		\]
		for all $x,x^{\prime}\in H$.
	\end{enumerate}
\end{assumption}

The $C^{1,1}$ regularity results of \cite{defeo_swiech_wessels_2025} were proved there for HJB equations with coefficient functions which were independent of $t$. Here the time variable is present through the variable $m(t)$. However it is easy to see that the presence of the time variable does not affect the proofs in \cite{defeo_swiech_wessels_2025} if all assumptions on the regularity of the coefficients are independent of $t$. Thus, combining Theorems 5.10 and 5.12 of \cite{defeo_swiech_wessels_2025} we have the following result. Note that the boundedness of the value function $v$, which is not claimed in \cite{defeo_swiech_wessels_2025}, follows from its stochastic representation and the assumption that the running and terminal costs are bounded, see Assumptions \ref{ass:HJB}(A5) and \ref{bsigmalipschitzw}(ii).

\begin{theorem}\label{th:C11reg1}
    Let Assumptions \ref{ass:HJB}(A1),(A4),(A5), \ref{bsigmalipschitzw}, \ref{bsigmafrechetw} be satisfied. Then, there exists a constant $\nu_0$, depending only on the data of the problem, such that if $\nu \geq \nu_0$, then for every $m\in S$ the viscosity solution $v$ of \eqref{HJB} satisfies
    \begin{equation}\label{eq:ex1regularity}
        \sup_{t\in[0,T]}\|v(t,\cdot)\|_{C^{1,1}_b(H_{-1})}\leq L
    \end{equation}
    for some $L\geq 0$ depending only on various constants in the assumptions and independent of $m$.
\end{theorem}

If the control enters the state equation linearly, one can show that $\nu_0$ can be made arbitrarily small by choosing $T>0$ sufficiently small, see Theorem \ref{appendix_theorem_semiconvex}. In that case, the previous theorem can be viewed as a short time regularity result.

\begin{remark}
    For a class of Hamiltonians $\mathcal{H}$ satisfying Assumptions \ref{ass:HJB}(A2),(A3), we refer to \cite[Section 7]{federico_gozzi_swiech_2026}. There, the drift $b(x,\alpha)$ and the running cost $f(x,\alpha,\mu)$ are assumed to have a separated structure in $(x,\mu)$ and $\alpha$. Therefore, the nonlinear dependence on $p$ in the Hamiltonian is separated from the remaining terms, and we obtain the Lipschitz continuity of $\mathcal{H}$ in $(x,\mu)$ with respect to the $\|\cdot\|_{-1}$ norm and $\mathbf{d}_{1,-1}$ metric, respectively, by imposing that Lipschitz continuity on the coefficients $b$ and $f$. The Lipschitz continuity of $\mathcal{H}_p$ follows exactly as in \cite[Section 7]{federico_gozzi_swiech_2026}.
\end{remark}

\begin{remark}\label{remark_lambda_unbounded}   
    Now, let us consider the case in which $\Lambda_0$ may be unbounded. In this case, in addition to Assumptions \ref{bsigmalipschitzw} and \ref{bsigmafrechetw}, we impose the following assumptions. The function $b:H \times \Lambda\times \mathcal{P}_1(H) \to H$ is of the form $b(x,\alpha,\mu) = b_1(x,\mu) + b_2(x,\alpha,\mu)$ for some $b_1:H \times \mathcal{P}_1(H) \to H$ and $b_2:H \times \Lambda\times \mathcal{P}_1(H) \to H$, and there is a constant $C\geq 0$ such that for all $x\in H$, $\alpha\in \Lambda_0$, $\mu\in \mathcal{P}_1(H)$
    \begin{equation}
        \| b_1(x,\mu) \| \leq C(1+\|x\|) ,\quad \|b_2(x,\alpha,\mu)\| \leq C(1 +\|\alpha\|_{\Lambda} ).
    \end{equation}
    The function $f:H\times \Lambda \times \mathcal{P}_1(H) \to \mathbb{R}$ is of the form $f(x,\alpha,\mu) = f_1(x,\mu) + f_2(x,\alpha,\mu)$ for some $f_1 : H\times \mathcal{P}_1(H) \to \mathbb{R}$ and $f_2:H\times\Lambda\times \mathcal{P}_1(H) \to \mathbb{R}$, and there are constants $C,C_1\geq 0$, and $C_2,C_3>0$ such that for all $x\in H$, $\alpha\in \Lambda_0$, $\mu\in \mathcal{P}_1(H)$
    \begin{equation}
        | f_1(x,\mu) | \leq C, \quad - C_1 + C_2 \| \alpha \|_{\Lambda}^2 \leq f_2(x,\alpha,\mu) \leq C_1 + C_3 \| \alpha \|_{\Lambda}^2.
    \end{equation}
    First, let us note that one can still argue as before to obtain the $C^{1,1}_b(H_{-1})$ regularity of the value function by using its stochastic representation (note that the assumptions in \cite{defeo_swiech_wessels_2025} are implied by the assumptions in the present case). Moreover, it follows from \cite[Theorem 5.5]{defeo_swiech_wessels_2025}, that the value function is Lipschitz continuous in the spatial variable with respect to the $\|\cdot\|_{-1}$ norm, uniformly in $t\in [0,T]$. Let us denote by $R>0$ the uniform Lipschitz constant of $V(t,\cdot)$ with respect to the $\|\cdot\|$ norm. Second, we observe that there is a constant $C_R > 0$ such that
    \begin{equation}
        - \langle b_2(x,\alpha,\mu), p \rangle - f_2(x,\alpha,\mu) \leq C (1+ \| \alpha \|_{\Lambda}) \| p \| + C_1 - C_2 \| \alpha \|_{\Lambda}^2 \leq C_R -\frac{C_2}{2} \| \alpha \|_{\Lambda}^2
    \end{equation}
    for all $x\in H$, $\alpha\in \Lambda_0$, $\mu\in \mathcal{P}_1(H)$ and $p\in H$ with $\|p\|\leq R$, and
    \begin{equation}
        \sup_{\alpha\in \Lambda_0} \left \{ - \langle b_2(x,\alpha,\mu), p \rangle - f_2(x,\alpha,\mu) \right \} \geq C.
    \end{equation}
    Thus, choosing $K>0$ sufficiently large, we obtain that
    \begin{equation}
        \mathcal{H}(x,p,\mu) = \sup_{\substack{\alpha\in \Lambda_0 \\ \| \alpha \|_{\Lambda} \leq K}} \{ - \langle b(x,\alpha,\mu) , p \rangle - f(x,\alpha,\mu) \}
    \end{equation}
    for all $x\in H$, $p\in H$ with $\|p\|\leq R$ and $\mu\in \mathcal{P}_1(H)$. Therefore, we can argue as in the proof of 
    \cite[Proposition 4.6]{defeo_swiech_wessels_2025} to get that $v=v_K$, where $v_K$ is the value function where the control set is restricted to the bounded set $\Lambda_K:=\{\alpha\in\Lambda_0:\| \alpha \|_{\Lambda} \leq K\}$. In particular $v$ is then also uniformly continuous in $(t,x)$ on bounded subsets of $[0,T]\times H$ and is the unique viscosity solution of \eqref{HJB} among bounded $B$-continuous functions which are Lipschitz continuous in $x$ and uniformly continuous on bounded subsets of $[0,T]\times H$.
\end{remark}

\subsection{Examples for the Uniqueness of the MFG System}

In this section, we discuss two examples of functions {$F$} that satisfy Assumption \ref{hp:uniqueness}(iii). The structure of the functions resembles that of the examples in \cite{federico_gozzi_swiech_2026} (see also \cite[Section 3.4.2]{carmona_delarue_2018}, when $H=\mathbb{R}^d$), however, note that here we need to work with the norm on $H_{-1}$ instead of the norm on $H$.

\begin{example}\label{example1}
Let $h_1:H \to\mathbb{R}$ and $h_2:H\to H$ be bounded and satisfy
\[
|h_1(x)-h_1(y)|+\|h_2(x)-h_2(y)\|\leq C\|x-y\|_{-1}
\]
for some $C\geq 0$ and all $x,y\in H$.
We define $F_1,F_2 : H \times \mathcal{P}_1(H) \to \mathbb{R}$ as
\[
F_1(x,\mu):=h_1(x)G_1\left(\int_Hh_1(y)\mu(\mathrm{d} y)\right),\quad F_2(x,\mu):=\left\langle h_2(x),G_2\left(\int_Hh_2(y)\mu(\mathrm{d} y)\right)\right\rangle,
\]
where $G_1:\mathbb{R}\to\mathbb{R}$ is Lipschitz continuous and strictly increasing and $G_2:H\to H$ is Lipschitz continuous and strictly monotone.
Then, recalling that $\mathcal{P}_1(H)\subset \mathcal{P}_1(H_{-1})$, if $\gamma\in \Gamma_{-1}(\mu_1,\mu_2)$,
\[
\begin{split}
|F_1(x_1,\mu_1)-F_1(x_2,\mu_2)|&\leq C_1\|x_1-x_2\|_{-1}+C_2\left|\int_{H_{-1}}h_1(y)\mu_1(\mathrm{d} y)-\int_{H_{-1}}h_1(z)\mu_2(\mathrm{d} z)\right|
\\
&
\leq C_1\|x_1-x_2\|_{-1}+C_3\int_{H_{-1}^2} \| y-z \|_{-1} \gamma(\mathrm{d}y,\mathrm{d}z).
\end{split}
\]
Taking the infimum over $\gamma \in \Gamma_{-1}(\mu_1,\mu_2)$ gives
\[
|F_1(x_1,\mu_1)-F_1(x_2,\mu_2)|\leq C_1\|x_1-x_2\|_{-1}+C_3\mathbf{d}_{1,-1}(\mu_1,\mu_2).
\]
A similar calculation gives the same result for $F_2$. Thus both functions satisfy Assumption \ref{ass:HJB}(A2).
We also have
\[
\begin{split}
\int_H [F_2(x,\mu_1)&-F_2(x,\mu_2)](\mu_1-\mu_2)(\mathrm{d}x)
\\
&=\left\langle\int_Hh_2(y)\mu_1(\mathrm{d}y)-\int_Hh_2(y)\mu_{2}(\mathrm{d}y),
G_2\left(\int_Hh_2(y)\mu_1(\mathrm{d} y)\right)-G_2\left(\int_Hh_2(y)\mu_2(\mathrm{d} y)\right)\right\rangle,
\end{split}
\]
and the corresponding identity also holds for $F_1$. Hence, Assumptions \ref{hp:uniqueness}(iii) and \ref{hp:uniqueness}(iv)(b) are satisfied for $F_1$ and $F_2$.

Regarding the $C^{1,1}(H_{-1})$ regularity of $F_2$, we observe that
\begin{equation}
    D_{H_{-1}} F_2(x,\mu) = \left ( D_{H_{-1}}h_2(x) \right )^* G_2\left(\int_H h_2(y) \mu(\mathrm{d}y)\right).
\end{equation}
Thus, if $D_{H_{-1}} h_2 : H_{-1} \to L( H_{-1} ,H)$ is Lipschitz continuous, then, for fixed $\mu \in \mathcal{P}_1(H)$, $F_2(\cdot,\mu) \in C^{1,1}(H_{-1})$. Similarly, for fixed $\mu\in \mathcal{P}_1(H)$, if $h_1 \in C^{1,1}(H_{-1})$, then $F_1(\cdot,\mu)\in C^{1,1}(H_{-1})$.

\end{example}

\begin{example}\label{ex:convolution}
Let $\ell:H\times[0,\infty)\to \mathbb{R}$ be bounded, suppose that $\ell(x,\cdot)$ is strictly increasing for every $x\in H$ and let
\[
|\ell(x,r)-\ell(y,s)|\leq C(\|x-y\|_{-1}+|r-s|)\quad\forall x,y\in H,r,s\in [0,\infty)
\]
for some $C\geq 0$. Let $\rho:H\to[0,\infty)$ be bounded and such that
\[
|\rho(x)-\rho(y)|\leq C\|x-y\|_{-1}\quad\forall x,y\in H
\]
for some $C\geq 0$. Let $\nu$ be a positive finite measure on $H$ with full support, i.e., $\nu(D)>0$ for every non-empty open set $D$ in $H$. We define
\[
F(x,\mu):=\int_H \ell(z,\rho\ast\mu(z))\rho(z-x)\nu(\mathrm{d}z),
\]
where $\rho\ast\mu$ denotes the convolution of $\rho$ with $\mu$, that is
\[
\rho\ast\mu(z)=\int_H \rho(z-u)\mu(\mathrm{d}u).
\]

Let us check Assumption \ref{ass:HJB}(A2). We first note that
\begin{equation}
    | \rho\ast\mu_1(z) - \rho\ast\mu_2(y) | \leq C ( \|z-y\|_{-1} + \mathbf{d}_{1,-1}(\mu_1,\mu_2) ), \quad \forall \mu_1,\mu_2 \in \mathcal{P}_1(H), z,y\in H.
\end{equation}
Indeed, let $\gamma\in \Gamma_{-1}(\mu_1,\mu_2)$. Then, considering $\mu_1,\mu_2$ as elements of $\mathcal{P}_1(H_{-1})$, we have
\begin{equation}
\begin{split}
    &| \rho \ast \mu_1(z) - \rho\ast\mu_2(y) |\\
    &= \left | \int_{H_{-1}} \rho(z-u) \mu_1(\mathrm{d}u) - \int_{H_{-1}} \rho(y-w) \mu_2(\mathrm{d}w) \right | = \left | \int_{H_{-1}^2} \left ( \rho(z-u) - \rho(y-w) \right ) \gamma(\mathrm{d}u,\mathrm{d}w) \right | \\
    &\leq C \int_{H_{-1}^2} \| ( z-u) - (y-w) \|_{-1} \gamma(\mathrm{d}u,\mathrm{d}w) \leq C \|z-y\|_{-1} + C \int_{H_{-1}^2} \| u-w \|_{-1} \gamma(\mathrm{d}u,\mathrm{d}w)
\end{split}
\end{equation}
for all $\gamma\in \Gamma_{-1}(\mu_1,\mu_2)$. Taking the infimum over $\gamma \in \Gamma_{-1}(\mu_1,\mu_2)$ yields the claim. Thus, we obtain
\begin{equation}
\begin{split}
    | F(x,\mu_1) - F(x,\mu_2) | &= \left | \int_H \ell (z,\rho \ast \mu_1(z)) \rho(z-x) \nu(\mathrm{d}z) - \int_H \ell (z,\rho \ast \mu_2(z)) \rho(z-x) \nu(\mathrm{d}z) \right |\\
    &\leq C \int_H \left | \ell (z,\rho\ast\mu_1(z)) - \ell(z,\rho\ast\mu_2(z)) \right | \nu(\mathrm{d}z)\\
    &\leq C \int_H \mathbf{d}_{1,-1}(\mu_1,\mu_2) \nu(\mathrm{d}z) = C \mathbf{d}_{1,-1}(\mu_1,\mu_2),
\end{split}
\end{equation}
which shows that $F$ satisfies Assumption \ref{ass:HJB}(A2). Assumptions \ref{hp:uniqueness}(iii) and \ref{hp:uniqueness}(iv)(b) are proved exactly as in \cite[Example 7.2]{federico_gozzi_swiech_2026}. Finally, note that
\begin{equation}
    D_{H_{-1}} F(x,\mu) = - \int_H \ell(z,\rho\ast\mu(z)) D_{H_{-1}} \rho(z-x) \nu(\mathrm{d}z).
\end{equation}
Therefore, if $\rho \in C^{1,1}(H_{-1})$, then $F(\cdot,\mu) \in C^{1,1}(H_{-1})$.
\end{example}

\section{Application to Optimal Advertising}\label{section_optimal_advertising} 

Extending classical mean field game models to incorporate memory effects naturally leads to dynamics governed by stochastic delay differential equations. Such models have been studied in the context of mean field control in \cite{defeo_gozzi_swiech_wessels_2025,fouque_zhang_2020,gozzi_masiero_rosestolato_2024,guatteri_masiero_wessels_2025} and mean field games in \cite{fouque_zhang_2018}. Note that the model in \cite{fouque_zhang_2018} is linear-quadratic. We consider here a mean field game of this type and show that it can be embedded within our general framework.

Given a flow of distributions $(m(s))_{s\in [t,T]}$, we consider the cost functional
\begin{equation}\label{cost_functional_delay_example}
\begin{split}
    J(t,x^0,x^1;\alpha(\cdot)) &= \mathbb{E} \left [ \int_t^T \left ( \nu | \alpha(s) |^2 - \frac{1}{1+ \exp(y(s))} + \int_{\mathbb{R}} \exp \left ( - \frac{(y(s) - z)^2}{2} \right ) m(s,\mathrm{d}z) \right ) \mathrm{d}s \right ]\\
    &\quad + \mathbb{E} \left [ - \frac{1}{1+ \exp(y(T))} + \int_{\mathbb{R}} \exp \left ( - \frac{(y(T) - z )^2}{2} \right ) m(T,\mathrm{d}z) \right ]
\end{split}
\end{equation}
subject to the dynamics
\begin{equation}\label{delay_SDE}
\begin{cases}
    \mathrm{d}y(s) = \left [ b_0 \left ( y(s) ,\int_{-d}^0 \eta(\theta) y(s+\theta) \mathrm{d}\theta \right ) + c_0 \alpha(s) \right ] \mathrm{d}s + \sigma_0(y(s)) \mathrm{d}W(s)\\
    y(t) = x^0 \in \mathbb{R}, \quad y(\theta) = x^1(\theta), \quad \theta \in [-d,0),
\end{cases}
\end{equation}
where $(x^0,x^1)\in \mathbb{R} \times L^2(-d,0)$ is the initial condition, and $d>0$ is the time horizon for the delay which enters the equation through some delay kernel $\eta \in W^{1,2}(-d,0)$. Moreover, $(W(s))_{s\in [0,T]}$ is a real-valued Brownian motion on some filtered probability space $(\Omega,\mathcal{F}, (\mathcal{F}_s)_{s\in [t,T]},\mathbb{P})$, $c_0 \in \mathbb{R}$, and $b_0: \mathbb{R}^2 \to \mathbb{R}$ and $\sigma : \mathbb{R} \to \mathbb{R}$ are Lipschitz continuous and bounded. Finally, the control $\alpha: [t,T]\times \Omega \to \Lambda_0 := [0,\infty)$ is assumed to be progressively measurable and square integrable, and we denote by $\mathcal{A}_t$ the set of admissible controls. For the precise setup, see \cite{defeo_federico_swiech_2024} or \cite{defeo_swiech_wessels_2025}.

In this model, the state variable $y(\cdot)$ represents the goodwill of a company, which is an indicator for the value of the brand taking into account the consumer awareness and the brand loyalty. The natural decay of the goodwill is modeled by the function $b_0$. The delay term takes into account memory effects in the consumer's perception of the brand. It is modeled with a delay kernel $\eta$. The control $\alpha\geq 0$ is the advertising expenditure and $c_0 >0$ is the effectiveness of advertising. Finally, we include noise in the equation to model random fluctuations in the goodwill.

In this mean field game, the representative company seeks to minimize the objective functional $J$ defined in \eqref{cost_functional_delay_example} and is competing with other companies through the distribution $m$. The logistic terms $-1/(1+ \exp(y(s)) )$ and $-1/(1+ \exp(y(T)))$ represent an intrinsic reward associated with the company's goodwill, which models diminishing returns and saturation. The interaction integrals penalize the company for having a level of goodwill similar to that of many other firms, reflecting the incentive to differentiate its brand and reduce direct competition for market share. Finally, the term $\nu |\alpha(s)|^2$, $\nu>0$, accounts for the cost of advertising efforts.

Let us now show how this problem can be rewritten to fit the setting in the present paper employing a well known lifting procedure to restore Markovianity, see e.g. \cite{defeo_federico_swiech_2024}. We introduce the Hilbert space $H= \mathbb{R} \times L^2(-d,0)$ with its natural inner product and define the following transport-type operator $A:\mathcal{D}(A) \subset H \to H$ via
\begin{equation}\label{definition_A}
    Ax = \begin{bmatrix} 0 \\ x_1' \end{bmatrix}, \quad \mathcal{D}(A) = \left \{ x = (x_0,x_1) \in H: x_1\in W^{1,2}(-d,0), x_1(0) = x_0 \right \}.
\end{equation}
Moreover, we introduce the functions $b: H\times \Lambda_0 \to H$ and $\sigma:H\to L(\mathbb{R},H)$ defined by
\[
    b(x,\alpha) = \begin{bmatrix} b_0 \left ( x_0, \int_{-d}^0 \eta(\theta) x_1(\theta) \mathrm{d}\theta \right ) + c_0 \alpha \\ 0 \end{bmatrix}, \quad x=(x_0,x_1) \in H, \alpha \in \Lambda_0.
\]
\[
\sigma(x)w=\begin{bmatrix} \sigma_0(x_0)w\\ 0 \end{bmatrix}, \quad x=(x_0,x_1) \in H, w\in\mathbb{R}.
\]
Given $x\in H$ and $\alpha(\cdot) \in \mathcal{A}_t$, we consider the infinite dimensional SDE
\begin{equation}\label{lifted_SDE}
\begin{cases}
    \mathrm{d}X(s) = [ AX(s) + b(X(s),\alpha(s)) ] \mathrm{d}s + \sigma(X(s)) \mathrm{d}W(s)\\
    X(t) = x.
\end{cases}
\end{equation}
We have the following correspondence between the finite dimensional delay equation \eqref{delay_SDE} and the lifted equation \eqref{lifted_SDE}, see \cite[Theorem 3.4]{federico_tankov_2015}.
\begin{proposition}
    Given $x\in H$ and $\alpha(\cdot) \in \mathcal{A}_t$, let $y^{x,\alpha}(\cdot)$ be the unique strong solution to \eqref{delay_SDE}, and let $X^{x,\alpha}(\cdot)$ be the unique mild solution to \eqref{lifted_SDE}. Then
    \begin{equation}
        X^{x,\alpha}(s) = ( y^{x,\alpha}(s), y^{x,\alpha}(s+\cdot)|_{[-d,0]} ) \quad \forall s\geq t.
    \end{equation}
\end{proposition}
This proposition in particular implies that the first marginal of the law of $X^{x,\alpha}(s)$ on $\mathbb{R}^n$ coincides with the law of $y^{x,\alpha}(s)$ for all $s\geq t$.
  
Moreover, we can rewrite the cost functional \eqref{cost_functional_delay_example} as follows: We introduce the functions $f_1:H \times \mathcal{P}_1(H) \to \mathbb{R}$ and $f_2: \Lambda_0 \to \mathbb{R}$ defined as
\begin{equation}
    f_1(x,\mu) = - \frac{1}{1+\exp(x_0)} + \int_{H} \exp \left ( - \frac{(x_0 - z_0)^2}{2} \right ) \mu(\mathrm{d}z) ,\quad f_2(\alpha)  = \nu |\alpha|^2,
\end{equation}
for $x = (x_0,x_1) \in H$, $\alpha\in \Lambda_0$, $\mu\in \mathcal{P}_1(H)$, and $G:H \times \mathcal{P}_1(H) \to\mathbb{R}$ defined as
\begin{equation}
    G(x,\mu) = - \frac{1}{1+\exp(x_0)} + \int_{H} \exp \left ( - \frac{(x_0-z_0)^2}{2} \right ) \mu(\mathrm{d}z),\quad x= (x_0,x_1) \in H, \mu\in \mathcal{P}_1(H).
\end{equation}
We have
\begin{equation}
    J(t,x;\alpha(\cdot)) = \mathbb{E} \left [ \int_t^T \left ( f_1(X(s),m(s)) + f_2(\alpha(s)) \right ) \mathrm{d}s + G(X(T),m(T)) \right ].
\end{equation}

Note that the operator $A$ defined in \eqref{definition_A} is not maximal dissipative. Therefore, we introduce the operator $\tilde{A}: \mathcal{D}(\tilde{A}) \subset H \to H$, $\tilde{A} := A- (x_0,0)$ and the drift $\tilde{b} : H \times \Lambda_0 \to H$,
\begin{equation}
    \tilde{b}(x,\alpha) = \begin{bmatrix} b_0 \left ( x_0, \int_{-d}^0 \eta(\theta) x_1(\theta) \mathrm{d}\theta \right ) + x_0 + c_0 \alpha \\ 0 \end{bmatrix}, \quad x=(x_0,x_1) \in H, \alpha \in \Lambda_0.
\end{equation}
and rewrite equation \eqref{lifted_SDE} as
\begin{equation}\label{lifted_SDE_rewritten}
\begin{cases}
    \mathrm{d}X(s) = [ \tilde{A}X(s) + \tilde{b}(X(s),\alpha(s)) ] \mathrm{d}s + \sigma(X(s)) \mathrm{d}W(s)\\
    X(t) = x.
\end{cases}
\end{equation}
Note that the mild solutions of \eqref{lifted_SDE} and \eqref{lifted_SDE_rewritten} coincide. Moreover, one can show that $\tilde{A}$ satisfies Assumption \ref{ass:HJB}(A1) with $B= (\tilde{A}^{-1})^* \tilde{A}^{-1}$, and it holds
\begin{equation}\label{estimate_delay_weak_norm}
    |x_0 | \leq \| x \|_{-1}, \quad \forall x=(x_0,x_1) \in H.
\end{equation}
For more details on the construction of $B$ and its properties in the present setting, see \cite[Section 3]{defeo_federico_swiech_2024}.

Next, let us show that Assumptions \ref{ass:HJB}(A2)--(A6) are satisfied. To this end, let us compute the Hamiltonian. We have for $x,p\in H$, $x=(x_0,x_1)$, $p=(p_0,p_1)$, and $\mu\in \mathcal{P}_1(H)$, $z=(z_0,z_1)$,
\begin{equation}\label{example_hamiltonian}
\begin{split}
    \mathcal{H}(x,p,\mu) &= \sup_{\alpha \in [0,\infty)} \left \{ - \langle \tilde{b}(x,\alpha), p \rangle - f_1(x,\mu) - f_2(\alpha) \right \}\\
    &= \sup_{\alpha \in [0,\infty)} \bigg \{ - b_0 \left ( x_0, \int_{-d}^0 \eta(\theta) x_1(\theta) \mathrm{d}\theta \right ) p_0 - x_0 p_0 - c_0 \alpha p_0\\
    &\qquad\qquad\qquad - \nu | \alpha |^2 + \frac{1}{1+\exp(x_0)} - \int_{H} \exp \left ( - \frac{(x_0 - z_0)^2}{2} \right ) \mu(\mathrm{d}z) \bigg \}\\
    &=: \mathcal{H}^0(x,p) - F(x,\mu),
\end{split}
\end{equation}
where $\mathcal{H}^0: H\times H\to \mathbb{R}$ is given by
\begin{equation}\label{H_0}
\begin{split}
    \mathcal{H}^0(x,p) = \begin{cases} 
        - b_0 \left ( x_0, \int_{-d}^0 \eta(\theta) x_1(\theta) \mathrm{d}\theta \right ) p_0 - x_0 p_0 + \frac{c_0^2 p_0^2}{4\nu}, & p_0 <0\\
        - b_0 \left ( x_0, \int_{-d}^0 \eta(\theta) x_1(\theta) \mathrm{d}\theta \right ) p_0 - x_0 p_0, & p_0 \geq 0.
    \end{cases}
\end{split}
\end{equation}
and $F:H\times\mathcal{P}_1(H) \to \mathbb{R}$ is given by
\begin{equation}
    F(x,\mu) = - \frac{1}{1+\exp(x_0)} + \int_{H} \exp \left ( - \frac{(x_0 - z_0)^2}{2} \right ) \mu(\mathrm{d}z).
\end{equation}
First, we note that we are in the setting of Remark \ref{remark_lambda_unbounded}. Indeed, the regularity needed in Assumptions \ref{bsigmalipschitzw} and \ref{bsigmafrechetw} follows as in \cite[Section 6.2]{defeo_swiech_wessels_2025}, see also \cite[Lemma 4.1]{defeo_federico_swiech_2024}. Hence, arguing as in the proofs of \cite[Propositions 4.6, 4.15]{defeo_swiech_wessels_2025}, we may restrict our attention to properties of the Hamiltonian $\mathcal{H}$ on $H \times B_R(0) \times \mathcal{P}_1(H)$, for some $R>0$. On this set, it follows easily using $\eta \in W^{1,2}(-d,0)$ and \eqref{estimate_delay_weak_norm} that $\mathcal{H}^0$ satisfies Assumption \ref{ass:HJB}(A2), see also \cite[Lemma 4.1]{defeo_federico_swiech_2024}. The Lipschitz continuity of $F$ with respect to the $\|\cdot\|_{-1}$ norm and $\mathbf{d}_{1,-1}$ metric, respectively, is established analogously as for the convolution term in Example \ref{ex:convolution}, invoking \eqref{estimate_delay_weak_norm}.
% For the first term, we use \eqref{estimate_delay_weak_norm}. For second term, we note that
%\begin{equation}
%\begin{split}
%    &\left | \int_{\mathbb{R}} \exp \left ( - \frac{(x_0 - z_0)^2}{2} \right ) \mu(\mathrm{d}z_0) - \int_{\mathbb{R}} \exp \left ( - \frac{(y_0 - z_0)^2}{2} 
%\right ) \mu'(\mathrm{d}z_0) \right |\\
  %  &\leq \left | \int_{\mathbb{R}} \left ( \exp \left ( - \frac{(x_0-z_0)^2}{2} \right ) - \exp \left ( - \frac{(y_0-z_0)^2}{2} \right ) \right ) \mu ( \mathrm{d}z_0) %\right | + \left | \int_{\mathbb{R}} \exp \left ( - \frac{(y_0-z_0)^2}{2} \right ) ( \mu - \mu' )(\mathrm{d}z_0) \right |\\
   % &\leq C \left ( | x_0 - y_0 | + \mathbf{d}_{1,-1}(\mu,\mu') \right ),
%\end{split}
%\end{equation}
%where we used the Kantorovich--Rubinstein formula to estimate the second term, see \cite[Particular Case 5.17]{villani_2009}. 
Thus, Assumption \ref{ass:HJB}(A2) is satisfied. Assumption \ref{ass:HJB}(A3) follows easily from \eqref{H_0} using \eqref{estimate_delay_weak_norm}. Assumptions \ref{ass:HJB}(A4) and (A5) also follow from \eqref{estimate_delay_weak_norm}.

In order to show that Assumption \ref{ass:HJB}(A6) is satisfied, we show that the viscosity solutions $v(t,\cdot)$, $v_n(t,\cdot)$ and $v_{n,N}(t,\cdot)$ of \eqref{HJB}, \eqref{HJBn} and \eqref{HJBnN}, respectively, are semiconcave and semiconvex in $H_{-1}$, uniformly in $n,N,t\in [0,T]$. To this end, note that all three functions can be interpreted as the value functions for corresponding control problems. Thus, we can apply the arguments from \cite[Section 5.2]{defeo_swiech_wessels_2025} to obtain the Lipschitz continuity and semiconcavity of these functions in $H_{-1}$ with constants independent of $n,N,t$. For the semiconvexity, we follow the same strategy as in \cite[Section 5.3.1]{defeo_swiech_wessels_2025} with a small refinement. Since in that reference, the control enters the state equation through a nonlinear function, the constant $\nu$ in the control cost needs to be greater than some constant $C$ that depends on the data of the problem. In the current case, the control enters the state equation linearly. Hence, we can apply the refined Theorem \ref{appendix_theorem_semiconvex} which states that $\nu$ needs to be greater than $C_1 \mathrm{e}^{C_2 T} T$ for some constants $C_1,C_2$ depending on the data of the problem but independent of $T$ (and here independent of $n,N$). Thus, this result can be viewed as a short time result: For a given parameter $\nu$ in the control cost, we obtain $C^{1,1}$ regularity for a sufficiently small time horizon $T$.

Finally, let us discuss the assumptions needed for uniqueness of the solution of the MFG system. It follows from \eqref{example_hamiltonian} and \eqref{H_0} that Assumption \ref{hp:uniqueness}(i) and (ii) are satisfied. Moreover, by \cite[Section 3.4.2, Example 5]{carmona_delarue_2018}, it is easy to see that 
\[ \int_H [ F(x,\mu_1) - F(x,\mu_2) ] ( \mu_1 - \mu_2)(\mathrm{d}x) 
\]
is non-negative and equal to zero if and only if the $x_0$-marginals of $\mu_1$ and $\mu_2$ coincide. Alternatively one can notice that the function $F(x,\mu)$ can be expressed as a version of the function $F_2$ from Example \ref{example1}.

%we have
%\begin{equation}
 %   \int_H [ F(x,\mu_1) - F(x,\mu_2) ] ( \mu_1 - \mu_2)(\mathrm{d}x) = \int_{\mathbb{R}} \int_{\mathbb{R}} \exp \left ( - \frac{(x_0 - z_0)^2}{2} \right ) %(\mu_1 - \mu_2)(\mathrm{d}z_0) (\mu_1 - \mu_2)(\mathrm{d}x_0) 
%\end{equation}
%which is non-negative and equal to zero if and only if the first marginals of $\mu_1$ and $\mu_2$ coincide, see \cite[Section 3.4.2, Example 5]
%{carmona_delarue_2018}. Thus, Assumption \ref{hp:uniqueness}(iii) and (iv)(b) are satisfied.

\appendix

\section{Viscosity Solutions of PDEs in Hilbert Spaces}\label{section_appendix}

We recall here the definition of a viscosity solution for a terminal value degenerate parabolic PDE on a Hilbert space with an unbounded operator from \cite[Section 3.3]{fabbri_gozzi_swiech_2017}.

Throughout this section $(V,\langle\cdot,\cdot\rangle)$ is a real separable Hilbert space. We denote by $S(V)$ the space of self-adjoint operators in $L(V)$.

\subsection{\texorpdfstring{$B$}{B}-Continuity}

Let $B \in S(V)$ be a strictly positive operator in $V$ and $I\subset\mathbb{R}$ be an interval.

\begin{definition}[$B$-continuity]\label{definition_b_continuity}
    Let  $u: I \times V \to \mathbb{R}$. We say that $u$ is $B$-upper semicontinuous (respectively, $B$-lower semicontinuous) if, for any sequences $\left(t_n\right)\subset I$ and $\left(x_n\right)\subset V$ such that $t_n \rightarrow t \in I,$ $ {x_n \rightharpoonup x }\in V$, $B x_n \rightarrow B x$ as $n \rightarrow \infty$, we have $\limsup _{n \rightarrow \infty} u\left(t_n, x_n\right) \leq u(t, x)$ (respectively,  $\liminf _{n \rightarrow \infty} u\left(t_n, x_n\right) \geq u(t, x)$). We say that $u$ is $B$-continuous if it is both $B$-upper semicontinuous and $B$-lower semicontinuous.
\end{definition}

\begin{remark}
    In the present paper the operator $B$ is assumed to be compact. In this case, a function is $B$-upper semicontinuous (respectively, $B$-lower semicontinuous, $B$-continuous) if and only if it is weakly sequentially upper semicontinuous (respectively, weakly sequentially lower semicontinuous, weakly sequentially continuous), see \cite[Lemma 3.6]{fabbri_gozzi_swiech_2017}.
\end{remark}

\subsection{Viscosity Solutions}

Consider the following terminal value PDE in the Hilbert space $V$
\begin{equation}\label{eq:PDE_app}
\begin{cases}
    \partial_t u+\langle A x, D u\rangle+\mathcal F\left(t, x, D u, D^2 u\right)=0, \quad (t,x)\in(0, T) \times V\\
    u(T, x)=g(x), \quad  x \in V,
\end{cases}
\end{equation}
where $A:\mathcal D(A)\subset V\to V$ is a linear densely defined maximal dissipative operator, and $\mathcal F: [0,T]\times V \times V \times S(V) \rightarrow \mathbb{R}$ and $g: V \rightarrow \mathbb{R}$ are continuous. We also assume that $\mathcal F\left(t, x, p, X\right)\leq \mathcal F\left(t, x, p, Z\right)$ for all $t\in[0,T],x,p\in V$ and $X,Z\in S(V), X\leq Z$. Let $B \in S(V)$ be strictly positive and such that $A^*B\in L(V)$.
\begin{definition}\label{def:test_functions_hilbert}
    A function $\psi \colon (0,T)\times V\to \mathbb R$ is a test function if $\psi = \varphi + h(t,\|x\|)$, where $\varphi \in C^{1,2}((0, T) \times V)$ is locally bounded and is such that $\varphi$ is $B$-lower semicontinuous, $\partial_t \varphi, A^* D \varphi, D \varphi, D^2 \varphi$ are uniformly continuous on $(0, T) \times V$, $h \in C^{1,2}((0, T) \times \mathbb{R})$ is such that for every $t \in(0, T), h(t, \cdot)$ is even and non-decreasing on $[0,+\infty)$.
\end{definition}

\begin{definition}\label{def:viscosity_solution_hilbert}
A locally bounded $B$-upper semicontinuous function $u:(0, T] \times V \to \mathbb{R}$ is a $B$-continuous viscosity subsolution of \eqref{eq:PDE_app} if $u(T, y) \leq g(y)$ for all $y \in V$ and whenever $u-\psi$ has a local maximum at a point $(t, x) \in(0, T) \times V$ for a test function $\psi(s, y)=\varphi(s, y)+h(s,\|y\|)$, then
\begin{equation}
\partial_t \psi(t, x)+\left\langle x, A^* D \varphi(t, x)\right\rangle+\mathcal F\left(t, x, D \psi(t, x), D^2 \psi(t, x)\right) \geq 0.
\end{equation}
A locally bounded $B$-lower semicontinuous function $u$ on $(0, T] \times V$ is a $ B$-continuous viscosity supersolution of \eqref{eq:PDE_app} if $u(T, y) \geq g(y)$ for $y \in V$ and whenever $u+\psi$ has a local minimum at a point $(t, x) \in(0, T) \times V$ for a test function $\psi(s, y)=\varphi(s, y)+h(s,\|y\|)$ then
\begin{equation}
    -\partial_t \psi(t, x)-\left\langle x, A^* D \varphi(t, x)\right\rangle+\mathcal F\left(t,x,-D \psi(t, x),-D^2 \psi(t, x)\right) \leq 0.
\end{equation}
A $B$-continuous viscosity solution of \eqref{eq:PDE_app} is a function which is both a $B$-continuous viscosity subsolution and a $ B$-continuous viscosity supersolution of \eqref{eq:PDE_app}.
\end{definition}

We remind the reader that here a function is called locally bounded if it is bounded on bounded sets.

\section{Semiconvexity of the Value Function}

In this appendix, we refine a semiconvexity result for the value function of a stochastic control problem originally established in \cite{defeo_swiech_wessels_2025}. This refinement accounts for the specific structure of the state equation where the control enters linearly and exclusively through the drift. In contrast, \cite{defeo_swiech_wessels_2025} treats the more general case where the control enters nonlinearly.
%Furthermore, for the specific estimate corresponding to Lemma \ref{estimatex1x0one}, the reference allows for the control to enter the noise %coefficient as well.

More precisely, we consider the value function
\begin{equation}\label{appendix_costfunctional}
	v(t,x) := \inf_{\alpha(\cdot)\in \mathcal{A}_t} \mathbb{E} \left [ \int_t^T f(s,X(s),\alpha(s)) \mathrm{d}s + G(X(T)) \right ]
\end{equation}
where the state is governed by
\begin{equation}\label{appendix_state_equation_nonlinear}
\begin{cases}
    \mathrm{d}X(s) = [ AX(s) + b(s,X(s)) + D \alpha(s) ] \mathrm{d}s + \sigma(s,X(s)) \mathrm{d}W(s), \quad s\in [t,T]\\
    X(t) = x\in H.
\end{cases}
\end{equation}
We use the same general setup as in Section \ref{sec:examples} and we impose the following assumptions on the coefficients:
\begin{assumption}\label{assumption_b_sigma_lipschitz}
\begin{enumerate}[label=(\roman*)]
    \item The function $b : [0,T]\times H \to H$ is continuous and there is a constant $C \geq 0$ such that
    \begin{equation}
        \| b(s,x) - b(s,x') \|_H \leq C \| x-x'\|_{-1},\quad \forall s\in [0,T], x,x'\in H.
    \end{equation}
    \item There is a constant $C\geq 0$ such that
    \begin{equation}
        \| b(s,x) \|_H \leq C(1+ \| x \|_H ),\quad \forall s\in [0,T], x\in H.
    \end{equation}
    \item The operator $D:\Lambda \to H$ is linear and bounded.
    \item The function $\sigma : [0,T]\times H \to L_2(\Xi,H)$ is continuous and there is a constant $C \geq 0$ such that
    \begin{equation}
        \| \sigma(s,x) - \sigma(s,x') \|_{L_2(\Xi,H)} \leq C \| x-x'\|_{-1},\quad \forall s\in [0,T], x,x'\in H.
    \end{equation}
    \item There is a constant $C\geq 0$ such that
    \begin{equation}
        \| \sigma(s,x) \|_{L_2(\Xi,H)} \leq C(1+ \| x \|_H ),\quad \forall s\in [0,T], x\in H.
    \end{equation}
\end{enumerate}
\end{assumption}

For $x_0,x_1\in H$ and $\alpha_0(\cdot),\alpha_1(\cdot) \in \mathcal{A}_t$, let $X_i(s) = X(s;t,x_i,\alpha_i(\cdot))$, $s\in [t,T]$, denote the solution of \eqref{appendix_state_equation_nonlinear} with initial condition $x_i$ and control $\alpha_i(\cdot)$, $i=1,2$.

\begin{lemma}\label{estimatex1x0one}
    Let Assumption \ref{assumption_b_sigma_lipschitz} be satisfied. Then, there are constants $C_1, C_2 \geq 0$ independent of $T$ such that
    \begin{equation}\label{difference_x1_x0}
    \begin{split}
        &\mathbb{E} \left [ \sup_{s\in [t,T]} \| X_1(s) - X_0(s) \|_{-1}^2 \right ]\\
        &\leq C_1 \mathrm{e}^{C_2 (T-t)} \left ( \| x_1 - x_0 \|_{-1}^2 + (T-t) \mathbb{E} \left [ \int_t^{T} \| \alpha_1(s) - \alpha_0(s) \|^2_{\Lambda} \mathrm{d}s \right ] \right ).
    \end{split}
    \end{equation}
\end{lemma}

\begin{proof}
    The proof follows along the same lines as the proof of \cite[Lemma 5.3, see also Lemma 3.3]{defeo_swiech_wessels_2025}. In the present case, there is no control in the noise coefficient $\sigma$. This allows us to obtain the additional $(T-t)$ term on the right-hand side of \eqref{difference_x1_x0}.
\end{proof}

\begin{assumption}\label{assumption_b_sigma_frechet}
	\begin{enumerate}[label=(\roman*)]
		\item Let $b:[0,T]\times H \to H$ be Fr\'echet differentiable in the second variable and let there be a constant $C>0$ such that
		\begin{equation}
			\| ( D b(s,x) - D b(s,x^{\prime}) )(x-x') \|_{-1} \leq C \|x-x^{\prime}\|^2_{-1}, \quad \forall s\in [0,T], x,x^{\prime}\in H.
		\end{equation}
		\item Let $\sigma:[0,T]\times H\to L_2(\Xi,H)$ be Fr\'echet differentiable in the second variable and let there be a constant $C>0$ such that
		\begin{equation}
			\| ( D \sigma(s,x) - D \sigma(s,x^{\prime}) ) (x-x') \|_{L_2(\Xi,H_{-1})} \leq C \|x-x^{\prime}\|^2_{-1}, \quad \forall s\in [0,T], x,x'\in H.
		\end{equation}
	\end{enumerate}
\end{assumption}

For $\lambda\in [0,1]$ we introduce
    \begin{equation}\label{appendix_lambdadefinition}
	\begin{cases}
		\alpha_{\lambda}(s) = \lambda \alpha_1(s) + (1-\lambda) \alpha_0(s)\\
		x_{\lambda} = \lambda x_1 +(1-\lambda)x_0\\
		X_{\lambda}(s) = X(s;t,x_{\lambda},\alpha_{\lambda}(\cdot))\\
		X^{\lambda}(s) = \lambda X_1(s) +(1-\lambda)X_0(s).  
	\end{cases}
    \end{equation}

\begin{lemma}\label{appendix_lemma_2}
    Let Assumptions \ref{assumption_b_sigma_lipschitz} and \ref{assumption_b_sigma_frechet} be satisfied. Then, there are constants $C_1,C_2\geq 0$ independent of $T$ such that
    \begin{equation}\label{appendix_estimate_lemma_2}
    \begin{split}
			&\mathbb{E} \left [ \sup_{s\in [t,T]} \| X^{\lambda}(s) - X_{\lambda}(s) \|_{-1} \right ]\\
            &\leq C_1 \mathrm{e}^{C_2(T-t)} \lambda (1-\lambda) \left ( \|x_1-x_0\|_{-1}^2 + (T-t) \mathbb{E} \left [ \int_t^T \|\alpha_1(s) - \alpha_0(s) \|_{\Lambda}^2 \mathrm{d}s \right ] \right )
    \end{split}
	\end{equation}
    for all $\lambda \in [0,1]$, $x_0,x_1\in H$ and $\alpha_0(\cdot),\alpha_1(\cdot) \in \mathcal{A}_t$.
\end{lemma}

\begin{proof}
    The proof follows along the same lines as the proof of \cite[Lemma 5.8, see also Lemma 3.9]{defeo_swiech_wessels_2025}. We just point out that in that proof, we now only need to estimate the simpler term
    \begin{equation}\label{estimateb}
        \lambda b(s,X_1(s)) + (1-\lambda) b(s,X_0(s)) - b(s,X^{\lambda}(s)).
    \end{equation}
    Since the control enters the state equation linearly, the corresponding term for the control vanishes. Therefore, the same proof as the one in \cite{defeo_swiech_wessels_2025} yields the additional factor $(T-t)$ on the right-hand side of \eqref{appendix_estimate_lemma_2}.
\end{proof}

\begin{assumption}\label{assumption_lglipschitzfirstvariablew}
\begin{enumerate}[label=(\roman*)]
    \item The function $f:[0,T]\times H\times\Lambda_0 \to \mathbb{R}$ is continuous and there exists a constant $C>0$ such that
    \begin{equation}
      |f(s,x,\alpha) - f(s,x^{\prime},\alpha)| \leq C \|x-x^{\prime}\|_{-1},\quad \forall s\in [0,T], x,x'\in H, \alpha\in \Lambda_0.
    \end{equation}
     \item There exists a constant $C>0$ such that
    \begin{equation}
      |f(s,x,\alpha)| \leq C (1+\|x\|_H+\|\alpha\|_{\Lambda}^2),\quad \forall s\in [0,T], x\in H, \alpha \in \Lambda_0.
    \end{equation}
    \item For $G:H\to \mathbb{R}$, there exists a constant $C>0$ such that
    \begin{equation}
        |G(x) - G(x^{\prime})| \leq C \|x-x^{\prime}\|_{-1}, \quad \forall x,x'\in H.
    \end{equation}
    \item If $\Lambda_0$ is unbounded, we assume that $f$ and $G$ are bounded from below.
\end{enumerate}
\end{assumption}

\begin{assumption}\label{assumption_lgsemiconvexw}
	\begin{enumerate}[label=(\roman*)]
		\item There exist constants $C,\nu \geq 0$ such that the map
		\begin{equation}
			H\times\Lambda_0 \ni (x,\alpha)\mapsto f(s,x,\alpha) + C \|x\|_{-1}^2 - \nu \|\alpha\|^2_{\Lambda}
		\end{equation}
		is convex for all $s\in [0,T]$.
    \item The function $G:H\to\mathbb{R}$ is semiconvex in $H_{-1}$.
	\end{enumerate}
\end{assumption}

\begin{theorem}\label{appendix_theorem_semiconvex}
    Let Assumptions \ref{assumption_b_sigma_lipschitz}, \ref{assumption_b_sigma_frechet}, \ref{assumption_lglipschitzfirstvariablew} and \ref{assumption_lgsemiconvexw} be satisfied. Then, there are constants $C_1, C_2\geq 0$ independent of $T$ such that if $\nu \geq \nu_0 := C_1 \mathrm{e}^{C_2T}T$, then $v(t,\cdot)$ is semiconvex uniformly in $t\in [0,T]$.
\end{theorem}

\begin{proof}
    The proof follows exactly the proof of \cite[Theorem 5.10, see also Theorem 3.11]{defeo_swiech_wessels_2025}. The presence of the additional factor $T$ in $\nu_0$ stems from Lemmas \ref{estimatex1x0one} and \ref{appendix_lemma_2}.
\end{proof}

\small{\paragraph{\textbf{Acknowledgments.}} Lukas Wessels acknowledges the financial support of the European Research Council (ERC) under the European Union’s Horizon Europe research and innovation programme (AdG ELISA project, Grant Agreement No. 101054746). Views and opinions expressed are however those of the authors only and do not necessarily reflect those of the European Union or the European Research Council Executive Agency. Neither the European Union nor the granting authority can be held responsible for them .}

\end{document}